\documentclass[11pt, a4paper]{article}
\usepackage[english]{babel}

\usepackage{amsmath}
\usepackage{amsthm}
\usepackage{amssymb}
\usepackage{enumerate}
\usepackage{dsfont}
\usepackage{verbatim}

\usepackage[dvips]{color}
\newtheorem{theorem}{Theorem}[section]

\newtheorem{lemma}[theorem]{Lemma}
\newtheorem{proposition}[theorem]{Proposition}
\newtheorem{corol}[theorem]{Corollary}
\newtheorem{definition}[theorem]{Definition}

\theoremstyle{definition}
\newtheorem{remark}[theorem]{Remark}

\newcommand{\si}{\mathbb{S}}

\newcommand{\rr}{\mathbb{R} }
\newcommand{\cc}{\mathbb{C}}
\newcommand{\hh}{\mathbb{H}}
\newcommand{\nn}{\mathbb{N}}

\newcommand{\meas}{\mathcal{M}}
\newcommand{\Z}{\mathcal{Z}}
\newcommand{\U}{\mathcal{U}}
\newcommand{\borel}{\mathbf{B}}
\renewcommand{\Re}{\mathrm{Re}}
\renewcommand{\Im}{\mathrm{Im}}
\newcommand{\ran}{\mathrm{ran}}
\newcommand{\closOP}{\mathcal{K}}
\newcommand{\boundOP}{\mathcal{B}}
\newcommand{\mon}{\mathcal{SR}}
\newcommand{\bimon}{\mathcal{N}}
\newcommand{\indi}{\mathds{1}}
\newcommand{\lap}{\mathcal{L}}
\newcommand{\id}{\mathcal{I} }
\newcommand{\R}{\mathbb{R}}
\newcommand{\C}{\mathbb{C}}
\renewcommand{\S}{\mathbb{S}}
\newcommand{\Res}{\mathrm{Res}}
\newcommand{\dom}{\mathcal{D} }

\usepackage{xcolor}

\title{\bf  Functions of the infinitesimal generator of a \\
strongly continuous  quaternionic group}

\author{Daniel Alpay\\
Department of Mathematics\\
Ben–Gurion University of the Negev\\
Beer-Sheva 84105 Israel \\
dany@math.bgu.ac.il
\\
\and
Fabrizio Colombo\\Politecnico di
Milano\\Dipartimento di Matematica\\Via E. Bonardi, 9\\20133 Milano,
Italy\\fabrizio.colombo@polimi.it
\\
\and
Jonathan Gantner \\
Politecnico di Milano\\
Dipartimento di Matematica\\
Via E. Bonardi, 9\\
20133 Milano,
Italy\\
jonathan.gantner@polimi.it
\\
\and
David P. Kimsey\\
(DPK)Department of Mathematics\\
Ben-Gurion University of the Negev\\
Beer-Sheva 84105 Israel\\
dpkimsey@gmail.com
}

\date{ }
\begin{document}
\maketitle
\begin{abstract}
The analogue of the Riesz-Dunford functional calculus has been introduced and studied recently as well as the theory of semigroups and groups of linear quaternionic operators.
In this paper we suppose that $T$ is the infinitesimal generator of a strongly continuous group of operators
$(\Z_T(t))_{t \in \rr}$ and we show how we can define bounded operators $f(T)$, where $f$ belongs to a class of functions which is larger than the class of slice regular functions,
using the quaternionic Laplace-Stieltjes  transform.
This class will include functions that are slice regular on the $S$-spectrum of $T$ but not necessarily at infinity.
Moreover, we establish the relation of $f(T)$ with the quaternionic functional calculus and we study the problem of finding the inverse of $f(T)$.

 \end{abstract}
\noindent AMS Classification: 47A10, 47A60.

\noindent {\em Key words}: Quaternionic infinitesimal generators, S-resolvent operator, S-spectrum, quaternionic  Laplace-Stieltjes  transform, quaternionic functional calculus, functions of the infinitesimal generator.

\section{Introduction}

In this paper we study the quaternionic counterpart following problems that naturally arise for groups or semigroups of
operators in complex Banach spaces.
With the recently introduced notion of $S$-spectrum $\sigma_S(T)$ for a quaternionic linear operators $T$, it was possible to develop the quaternionic functional calculus (also called $S$-functional calculus) which is the quaternionic version of the Riesz-Dunford functional calculus, see \cite{acgs, JGA, CLOSED} and the book \cite{MR2752913} (see \cite{ds, rudin} for the classical case).
Moreover, the slice continuous functional calculus for quaternionic normal linear operators on a Hilbert space has been developed in \cite{GMP} using this new notion of spectrum.
The spectral theorem (see \cite{ds2} for the classical case) for quaternionic linear operators has been recently proved in \cite{ack, acks2, spectcomp}.
In the literature there were several attempts to prove the spectral theorem but the notion of spectrum was not made clear, see \cite{sc, Viswanath}. In \cite{fp} the authors prove the spectral theorem  for a quaternionic matrix $M$
using the notion of right spectrum $\sigma_R(M)$ which turned out to be equal to the $S$-spectrum $\sigma_S(M)$.

Quaternionic operators are not just interesting from a mathematical point of view,
they are important in the formulation of quantum mechanics (Q.M.). In fact, it was proved by
G. Birkhoff and  J. von Neumann \cite {BvN}, that there are essentially
just two possible ways to formulate Q.M.: the well known one using complex numbers and also one using quaternions.
The second description of Q.M. has been investigated by several authors, see \cite{adler, 12, 14, 21}, but the correct notion of spectrum did not seem to be known.
In the past,  the notion of right spectrum $\sigma_R(T)$ of a quaternionic linear operator $T$
was used successfully in some cases. This is because
$\sigma_R(T)$ coincides with $S$-point spectrum.

Classical results on groups and semigroups of linear operators (see
\cite{EngelNagel, Hille, Kantorovitz, Lunardi})
have been extended
to the quaternionic setting in recent papers:
it has been shown that the Hille-Yosida theorem holds, see \cite{MR2803786},
and it has been studied the problem of generation by perturbations of the quaternionic infinitesimal generator
in \cite{perturbation}.
For semigroups over real alternative *-algebras generation theorems and spherical sectorial operators have been studied in \cite{GR}.

We now recall some facts in the classical case, for more details see  \cite{ds}.
The Riesz-Dunford functional calculus can be extended to unbounded closed operators
$A: \mathcal{D}(A)\subset X\to X$, with non void resolvent set $\rho(A)$, where $X$ is a complex Banach space, see p. 599 in \cite{ds}.
Precisely, for each  function $f$ holomorphic on the spectrum $\sigma(A)$ and at infinity, we define the bounded linear operator
$f[A]$ as follows: consider the complex sphere $\mathbf{K}$ and the homeomorphism $\Phi:\mathbf{K}\to \mathbf{K}$ defined by
$\mu=\Phi(\lambda)=(\lambda-\mu)^{-1},\ \ \ \Phi(\infty)=0,\ \ \Phi(\mu)=\infty$ and the relation $\varphi(\mu):=f(\Phi^{-1}(\mu))$.
Let   $\mathcal{I}_X$ be the identity operator on $X$, it can be shown that the bounded operator
$$
f[A]:=\varphi((A-\mu \mathcal{I}_X)^{-1}), \ \ \ {\rm for}\ \ \ \ \mu\in \rho(A)
$$
is well defined, since it does not depend on  $\mu\in \rho(A)$, and  can be represented as
$$
f[A]=f(\infty)\mathcal{I}_X+\int_\Gamma (\lambda \mathcal{I}_X-A)^{-1} f(\lambda) d\lambda
$$
where $f(\infty):=\lim_{\lambda\to \infty}f(\lambda)$ and
$\Gamma$ is a suitable curve that surrounds the spectrum.

When $A$ is the infinitesimal generator of a strongly continuous group $U_A(t)$ for $t\in \mathbb{R}$, using the
bilateral Laplace-Stieltjes  transform it is possible to define the bounded linear operator $f(A)$ 
 on a larger class of functions.
In fact if $\alpha$ denotes a finite complex-valued measure defined on $\mathbb{R}$ so that the integral $\int_{\mathbb{R}} e^{(\omega+\varepsilon)t} d|\alpha|(t)$ is finite,  here $|\alpha|$ is the total variation of the measure, we can define the
bilateral Laplace-Stieltjes  transform of $\alpha$ as
$$
f(\lambda)=\int_{\mathbb{R}} e^{-t\lambda } d\alpha(t), \ \ \ -(\omega+\varepsilon)< \Re(\lambda)<\omega+\varepsilon.
$$
It turns out that the operator
$$
f(A)x:=\int_{\mathbb{R}} U_A(-t)x \,d\alpha(t),\ \ \ x\in X
$$
is well defined and it is bounded.
When the function  $f$ is holomorphic also at the infinity then the operator $f(A)$, define by the
Laplace-Stieltjes transform, and the one defined by the Riesz-Dunford functional calculus $f[A]$ are the same.

Under some conditions the inverse of the operator  $f(A)$ can be obtained using a suitable sequence of polynomials. If $p_n(\lambda)$ is such that $\lim_{n\to\infty}p_n(\lambda)f(\lambda)=1$,
then   $\lim_{n\to\infty}p_n(A)$ defines the inverse of $f(A)$.

To  extend the above results to the quaternionic setting we have to face several difficulties that are explained in the following sections.
Here we point out that in the quaternionic setting there are two resolvent operators and the notion  of
holomorphicity has to be replaced by the notion of slice hyperholomorphicity (or slice regularity)
which is recalled in the next section for the sake of convenience. For more details see the books
 \cite{MR2752913, GSSb} and the paper \cite{GP} for a different approach.

Before we explain our results in the quaternionic setting we point out another crucial fact.
We restrict to the case of bounded operators for sake of simplicity,
 but what follows holds also for unbounded operators.
The relation between  the resolvent operator $(\lambda \mathcal{I}_X-A)^{-1}$ of the  infinitesimal generator  $A:X\to X$
of a semigroup $(e^{tA})_{t \geq 0}$ is given by the Laplace transform
$$
(\lambda \mathcal{I}_X-A)^{-1}=\int_0^\infty e^{-t\lambda} e^{tA}\, dt,
$$
 for ${\rm Re}(\lambda)$ sufficiently large. This important relation  holds also in the quaternionic setting, but
 in this case we have two resolvent operators.
 Precisely, we define the $S$-spectrum of the bounded quaternionic linear operator $T$ as
 $$
 \sigma_S(T)=\{s\in \mathbb{H} \ : \   T^2-2s_0T+|s|^2\mathcal{I} \  \ {\rm is \ not\ invertible\ in} \ \mathcal{B}(V) \}
 $$
 where $\mathcal{B}(V)$ denotes the space of all bounded linear operators on a bilateral quaternionic Banach space $V$ while
  $s_0$ and $|s|$ are the real part and the modulus of the quaternion $s=s_0+s_1i+s_2j+s_3k$, respectively.
 The $S$-resolvent set $\rho_S(T)$ is defined by
$
\rho_S(T)=\mathbb{H}\setminus\sigma_S(T).
$
\noindent
For $s\in \rho_S(T)$ we define the left $S$-resolvent operator as
$$
S_L^{-1}(s,T):=-(T^2-2s_0 T+|s|^2\mathcal{I})^{-1}(T-\overline{s}\mathcal{I}),
$$
and the right $S$-resolvent operator as
$$
S_R^{-1}(s,T):=-(T-\overline{s}\mathcal{I})(T^2-2s_0 T+|s|^2\mathcal{I})^{-1}.
$$
For $T\in \mathcal{B}(V)$ and  $s_0 >\|T\|$ we have the relations
$$
S_R^{-1}(s,T)=\int_0^{\infty} e^{-t \,s}\, e^{t\,T}\, dt,
\ \ \ \ \ \ \ \
S_L^{-1}(s,T)=\int_0^{\infty} e^{tT}\, e^{-t s}\, dt.
$$
The above relations hold for right linear as well as for left linear quaternionic operators, but in the case of unbounded operators
the definitions of both resolvent operators have to be modified in order that they are defined on the whole quaternionic Banach space $V$.
Due to technical reasons, explained in Remark \ref{RKREASON}, we will only consider the right S-resolvent operator.

To give the definition and to prove the properties of the functions of the quaternionic infinitesimal generator of a strongly continuous group, it is necessary to have
all the preliminary results on quaternionic measure theory that we collect in Section \ref{sec3}.
We can now state the main results of this paper.
We suppose that $T$ is the quaternionic infinitesimal generator of the strongly continuous group $(\Z_T(t))_{t \in \rr}$ on a quaternionic Banach space $V$. For $f$
in the set of quaternionic bilateral Laplace-Stieltjes transforms of measures  with
\[
f(s)=\int_{\rr }d\mu(t)\,e^{-st }  \quad\text{for } -(\omega+\varepsilon)<\Re(s)\leq \omega+\varepsilon,
\]
where $\mu\in \mathbf{S}(T)$ (see in the sequel), we define the right linear operator $f(T)$ on $V$ by
\begin{equation}
f(T)v=\int_{\rr } d\mu(t)\, \Z_T(-t)v\qquad\text{for }v\in V.
\end{equation}
The operator $f(T)$ is bounded and under suitable conditions on $f$ and $g$ it has the property  $(fg)(T)=f(T)g(T)$.

In Theorem \ref{CompSCalc} we have proved that
if $f\in \mathbf{V}(T)$ and  $f$ is right slice regular at infinity,
 then the operator $f(T)$ defined using the Laplace transform equals the operator $f[T]$ obtained from the $S$-functional calculus.
\\
 Finally we  deduce sufficient conditions such that
 $$
 \lim_{n\to\infty}P_n(T)f(T)u=u, \ \ \text{ for  every $u\in V$}
 $$
 where $P_n$ are suitable polynomials.
 We conclude by recalling that
 there are several applications of the $S$-resolvent operators in
 Schur analysis, in particular in  the realization of Schur functions in the slice hyperholomorphic setting,
  see \cite{acs1, acs2, acs3},
and see  \cite{MR2002b:47144, adrs} for Schur analysis in the holomorphic case.

\section{Preliminary results on quaternionic operators}\label{S2}

The skew field of quaternions is defined as the real vector space $\hh = \{x=\sum_{i=0}^3\xi_ie_i: \xi_i\in\rr\}$ endowed with a multiplication such that $e_0 = 1$ is the identity, $e_i^2 = -1$ and $ e_ie_j =  - e_je_i$ for $i,j \in\{1,2,3\}$ with $i\neq j$. The real and imaginary part, the conjugate and the modulus of a quaternion are defined analogously to the case of complex numbers as
\[ \Re(x) = \xi_0,\qquad \Im(x) =\underline{x}=\sum_{i=1}^3\xi_ie_i,\qquad \overline{x} = \Re(x) - \underline{x}\qquad\text{and} \qquad|x| = \sqrt{\sum_{i=0}^3\xi_i^2}. \]
The quaternions also possess a complex structure. In order to explain it we define the set of imaginary units: We denote by $\S$ the set of all purely imaginary unit vectors, that is
\[\S:=\left \{\sum_{i=1}^3\xi_ie_i\in\hh:\sum_{i=1}^3\xi_i^2 = 1\right\}.\]
Given an element $x=\Re(x)+\underline{x}\in\hh$, we set
\[
I_x:=\begin{cases}\underline{x}\:\!/\:\!|\underline{x}| & \text{if }\underline{x}\neq0\\
\text{any element of $\S$} & \text{if } \underline{x} = 0.\end{cases}
\]
Then $x = x_0 + I_xx_1$ with $x_0 = \Re(x) = \xi_0$ and $x_1 = |\underline{x}|$. For any element $x=x_0+ I_xx_1\in\hh$, the set
$$[x]:= \{x_0 + Ix_1: I\in\S\}$$
is a 2-sphere in $\hh$, which degenerates to a single point if $\underline{x} = 0$.

For $I\in\S$, we obviously have $I^2=-1$. Hence, the vector space $\C_I = \mathbb{R}+I\mathbb{R}$ passing through $1$ and $I\in \mathbb{S}$ is isomorphic to the field of complex numbers. Moreover, if $I,J\in\S$ with $I\perp J$, then the quaternions $1$, $I$, $J$ and $IJ$ form an orthogonal basis of $\hh$ as a real vector space and so $\hh = \cc_I + \cc_IJ$. Furthermore, since  $IJ = - JI$, $z J = J\overline{z}$ for any $z\in\cc_I$. Hence, we also have $\hh = \cc_I + \cc_I J$.

\subsection{Slice regular functions}
The functional calculus considered in this paper is based on the theory of slice regular functions. We give a short introduction and state its most important results; the proofs can be found in \cite{MR2752913}.
Let $U\subset\hh$ be an open set and let
$f: U\to\hh$ be a real differentiable function. For
$I\in\mathbb{S}$ denote by $f_I$ be the restriction of $f$ to the
complex plane $\cc_I = \{x_0 + I x_1: x_0,x_1\in\rr\}$ and define the differential operator $\partial_I$ as
\[\partial _I := \frac12\left(\frac{\partial}{\partial x_0} + I\frac{\partial }{\partial x_1}\right).\]

\begin{definition}[Slice regular functions]{\rm
\label{defsmon}

The function  $f$ is said to be left slice regular if, for every
 $I\in\mathbb{S}$, it satisfies
\[
\partial_If(x) = \frac{1}{2}\left(\frac{\partial }{\partial x_0}f_I(x)+I\frac{\partial}{\partial x_1}f_I(x)\right)=0
\]
for all $x = x_0 + Ix_1\in U\cap \mathbb{C}_I$. We denote the set of left slice regular functions on  $U$  by $\mon^L(U)$.

The function $f$ is said to be right slice regular if,
for every$I\in\mathbb{S}$, it satisfies
\[ (f\partial_I)(x) =  \frac{1}{2}\left(\frac{\partial }{\partial x_0}f_I(x)+\frac{\partial}{\partial x_1}f_I(x)I\right)=0\]
for all $x= x_0 + Ix_1\in U\cap \mathbb{C}_I$.
We denote the set of right slice regular functions on $U$  by $\mon^R(U)$.}
\end{definition}

Any power series of the form $\sum_{n=0}^\infty x^n a_n$ with $a_n\in\hh$ for $n\in \mathbb{N}$ is left slice regular and any power series of the form $\sum_{n=0}^\infty b_nx^n$ with $b_n\in\hh$ is right slice regular. Conversely, at any real point, any left or right slice regular function allows a power series expansion of the respective form.

In the present paper we mainly consider right slice regular functions. For this reason we discuss in detail only the theory of right slice regular functions, although corresponding results also hold true for left slice regular functions.
\begin{definition}\label{SliceDeriv}{\rm
Let $U\subset\hh$ be open. The slice derivative of a function $f\in\mon^R(U)$ is the function defined by
\[ \partial_s f (x) := \frac12\left(\frac{\partial}{\partial x_0} f_I(x_0 + Ix_1) - \frac{\partial}{\partial x_1}f_I(x)I\right)\quad\text{for } x = x_0 + I x_1 \in U.\]}
\end{definition}
\begin{corol}
Let $U\subset\hh$ be open. If $f\in\mon^R(U)$, then $\partial_s f\in\mon^R(U)$.
\end{corol}

Note that the slice derivative coincides with the partial derivative with respect to the real part since
\begin{align*}\partial_s f (x) = \frac12\left(\frac{\partial}{\partial x_0} f_I(x) - \frac{\partial}{\partial x_1}f_I(x)I\right)= \frac12\left(\frac{\partial}{\partial x_0} f_I(x) + \frac{\partial}{\partial x_0}f_I(x)\right) = \frac{\partial}{\partial x_0}f_I(x)
\end{align*}
for any right slice regular function.

\begin{lemma}
Let $\alpha \in\R$  and  $B_r(\alpha)$ be the open ball of radius $r$ centered at $\alpha$. A function $f: B_r(\alpha)\to\hh$ is right slice regular if and only if
\[f (x) = \sum_{n=0}^{\infty} \frac{1}{n!}\partial_s^nf(\alpha) \,(x-\alpha)^n . \]
\end{lemma}

\begin{lemma}[Splitting Lemma]\label{SplitLem}
Let $U\subset\hh$ be open. A real differentiable function $f:U\to\hh$ is right slice regular if and only if for all $I,J\in\S$ with $I\perp J$ there exist holomorphic functions $f_1,f_2: U\cap\cc_I \to\cc_I$ such that
\[f_I (x) = f_1(x) + J f_2(x)\quad\text{ for all }x\in U\cap\cc_I.\]

\end{lemma}

Slice regular functions possess good properties when they are defined on suitable domains which are
introduced in the following definition.

\begin{definition}[Axially symmetric slice domain]\label{axsymm}{\rm
Let $U$ be a domain in $\hh$.
We say that $U$ is a
\textnormal{slice domain} if $U \cap \mathbb{R}$ is nonempty and if $U\cap \mathbb{C}_I$ is a domain in $\mathbb{C}_I$ for all $I \in \mathbb{S}$.
We say that $U$ is
\textnormal{axially symmetric} if, for all $x \in U$, the
$2$-sphere $[x]=x_0 + \S x_1 $ is contained in $U$.}
\end{definition}

\begin{theorem}[Representation Formula]\label{RepFo} Let $U\subset\hh$ be an axially symmetric slice domain. Let $I\in\S$ and set $x_I = x_0 + I x_1$ for $x = x_0 + I_xx_1\in\hh$.
If $f$ is a right slice regular function on $U$, then
\begin{align*}\label{distributionright}
f(x) & = f(x_I)(1-II_x)\frac12 + f(\overline{x_I})(1+II_x)\frac12\\
&=\frac{1}{2}\big[   f(x_I)+f(\overline{x_I})\big] +\frac{1}{2}\big[f(\overline{x_I})- f(x_I)\big]II_x
\end{align*}
for $x\in U$. Moreover, the quantities
\begin{equation}\label{capparight}
\alpha(x_0,x_1):=\frac{1}{2}\big[   f(x_I)+f(\overline{x_I})\big] \qquad\text{and}
\qquad  \beta(x_0,x_1):=\frac{1}{2}\big[f(\overline{x_I}) - f(x_I)\big]I
\end{equation}
are independent of the imaginary unit $I\in\S$.
\end{theorem}
The Representation Formula (see Theorem \ref{RepFo}) for left slice regular functions reads as
\begin{equation}\label{RepFoL} f(x) = \frac12 (1-I_xI)f(x_I) + \frac12(1+I_xI)f(\overline{x_I}).\end{equation}

Slice regular function satisfy a version of Cauchy's integral formula with a modified kernel.
\begin{definition}{\rm
For $x\notin [s]$ we define the noncommutative right Cauchy kernel as
\[
S_R^{-1}(s,x):= -(x-\bar s)(x^2-2{\rm Re}(s)x+|s|^2)^{-1}.
\]
and the  noncommutative left Cauchy kernel as
\[S_L^{-1}(s,x):=-(x^2 -2 \Re(s)x+|s|^2)^{-1}(x-\overline s).\]}
\end{definition}

\begin{lemma} The noncommutative Cauchy kernels have the following properties:
\begin{enumerate}[(i)]
\item The right Cauchy kernel $S_R^{-1}(s,x)$  is left slice regular in the variable $s$ and right slice regular in the variable $x$.
\item It holds the identity
\[S_L^{-1}(s,x) = - S_R^{-1}(x,s).\]
\item They are the Laplace transforms of the exponential function. For $\Re(s) > \Re(x)$, it is
\[S_R^{-1}(s,x) = \int_{0}^\infty e^{-st}e^{xt}\, dt\quad \text{and}\quad S_L^{-1}(s,x) = \int_{0}^\infty e^{xt} e^{-st}\, dt\]
and for $\Re(s) < \Re(x)$, it is
\[S_R^{-1}(s,x) =- \int_{-\infty}^0 e^{-st}e^{xt}\, dt\quad \text{and}\quad S_L^{-1}(s,x) = -\int_{-\infty}^0 e^{xt} e^{-st}\, dt.\]
\end{enumerate}
\end{lemma}

The right Cauchy kernel is the right slice regular inverse of the function $x\mapsto s-x$. This motivates the notation $S_R^{-n}(s,x)$ for the inverse of the $n$-th slice regular power of the map $x\mapsto s-x$.  (For details on the slice regular product, we refer again to \cite{MR2752913}.)
\begin{definition}\label{SlicePow}{\rm
Let $n\in\nn$. For $x\notin[s]$ we define
\[ S_R^{-n}(s,x) = \sum_{k=0}^{n} \binom{n}{k} s^{n-k}(-x)^{k} (x^2-2\Re(s)x + |s|^2)^{-n}.\]}
\end{definition}
Note that, for $m\in\nn$, we have
\[\frac{\partial^m}{\partial s_0^m} S_R^{-1}(s,x) = (-1)^m m!\, S_R^{-(m+1)}(s,x)\]
and
\begin{equation}\label{CauchyDeriv}
\frac{\partial^m}{\partial x_0^m} S_R^{-1}(s,x) = m! \,S_R^{-(m+1)}(s,x).
\end{equation}

The noncommutative Cauchy kernels allow us to prove the slice regular version of Cauchy's integral formula, which is the starting point for the definition of the $S$-functional calculus for quaternionic linear operators.

\begin{theorem}[The Cauchy formula with slice regular kernel]
\label{Cauchygenerale}
Let $U\subset\hh$ be an axially symmetric slice domain such that $\partial (U\cap \mathbb{C}_I)$ is a finite union of
continuously differentiable Jordan curves  for every $I\in\mathbb{S}$ and  set  $ds_I=-ds I$ for $I\in \mathbb{S}$.
If $f$ is right slice regular on a set that contains $\overline{U}$,
then
\begin{equation}\label{Cauchyright}
 f(x)=\frac{1}{2 \pi}\int_{\partial (U\cap \mathbb{C}_I)}  f(s)\,ds_I\, S_R^{-1}(s,x) \qquad\text{for all }x\in U.
 \end{equation}
This integral neither depends on $U$ nor on the imaginary unit $I\in\mathbb{S}$.
\end{theorem}

Finally, we introduced an important subclass of slice regular functions.
\begin{definition}{\rm
Let $U\subset\hh$ be an axially symmetric slice domain. A function $f\in\mon^R(U)$ is called intrinsic if $f(\overline{x}) = \overline{f(x)}$ for all $x\in U$. We denote the set of all intrinsic functions by $ \bimon(U) $.}
\end{definition}
The class $\bimon(U)$ plays a privileged role within the set of slice regular functions.  It contains all power series with real coefficients and any function that belongs to it is both left and right slice regular. We give two equivalent characterizations of intrinsic functions.
\begin{corol}Let $U\subset\hh$ be an axially symmetric slice domain and let $f\in \mon^R(U)$.
\begin{enumerate}[(i)]
\item The function $f$ belongs to $\bimon(U)$ if and only if $f(U\cap\cc_I)\subset\cc_I$ for all $I\in\S$.
\item Write $f(x)$ as $f(x) = \alpha(x_0,x_1)+ \beta(x_0,x_1)I_x$ according to Theorem~\ref{RepFo}. Then $f$ belongs to $\bimon(U)$ if and only if $\alpha$ and $\beta$ are real-valued.
\end{enumerate}
\end{corol}

Finally, observe that the product of two right slice regular functions is in general not slice regular.  The special role of intrinsic functions is therefore due to the following observation.
\begin{lemma}
Let $f,g\in\mon^R(U)$. If $g$ is intrinsic then $fg$ belongs to $\mon^R(U)$.
\end{lemma}

\subsection{The $S$-resolvent operator and the $S$-functional calculus}
We consider right linear operators on a two-sided quaternionic Banach space $V$.  We denote the set of all bounded right linear operators on $V$ by $\boundOP(V)$ and we define the set $\closOP (V)$ to be the set of closed right linear operators whose domain is dense in $V$. Based on the theory of slice regular functions it is possible to define a functional calculus for such operators. It is the natural generalization of the Riesz-Dunford-functional calculus for complex linear operators to the quaternionic setting; for details see again \cite{MR2752913}.

If $T$ is a closed operator with dense domain, then $T^2-2 \Re(s)T+|s|^2\id :{\cal D}(T^2)\subset V\to V$ is closed.
\begin{definition}{\rm
Let $T$ be a quaternionic  right linear operator and let $\mathcal{R}_s(T): \mathcal{D}(T^2) \to V$ be given by
$$
\mathcal{R}_s(T)x : = (T^2 - 2 {\rm Re}(s) T + |s|^2 \mathcal{I} )x, \ \ \ \  x \in \mathcal{D}(T^2).
$$
The $S$-resolvent set of $T$ is defined as follows
$$
\rho_S(T):=  \{  s\in V \ \  :\ \  \ker(\mathcal{R}_s(T))=\{0\}, \ \   \ran(\mathcal{R}_s(T)) \ \ is\ \  dense \  \ in  \ \ V  \ \ and
\ \  \mathcal{R}_s(T)^{-1} \in \mathcal{B}(V) \},
$$
where $\mathcal{R}_s(T)^{-1}: \ {\rm Ran}(\mathcal{R}_s(T))\to \mathcal{D}(T^2)$.
The $S$-spectrum of $T$ is defined as
$$
\sigma_S(T):=\mathbb{H}\setminus \rho_S(T).
$$}
\end{definition}

\begin{lemma}
The $S$-spectrum of an operator $T\in\closOP(V)$ is axially symmetric. If $T$ is bounded then $\sigma_S(T)$ is a nonempty, compact set contained in the closed ball $\overline{B_{\|T\|}(0)}$.
\end{lemma}
\begin{definition}[The right $S$-resolvent operator]{\rm
Let $T\in \closOP (V)$. For $s\in \rho_S(T)$ we define the right $S$-resolvent operator of $T$ as
\begin{equation}\label{SRDefi}
S_R^{-1}(s,T):=-(T-\overline{s}\id )(T^2 - 2\Re(s) T + |s|^2\id)^{-1}.
\end{equation}}
\end{definition}

\begin{remark} Observe that since the operator $(T^2 - 2\Re(s)T + |s|^2\id )^{-1}: V\to \dom(T^2)$ is bounded, the operator
$S_R^{-1}(s,T) =  -(T-\overline{s}\id)(T^2 - 2\Re(s)T + |s|^2\id )^{-1}$ is bounded.
\end{remark}
\begin{remark}\label{RKREASON}
The definition of the right $S$-resolvent operator is obviously inspired by the right Cauchy-kernel of slice hyperholomorphic functions. Analogously it is possible to define the left $S$-resolvent operator $S_L^{-1}(s,T)$ which inspired by the left Cauchy-kernel. However, since we are considering right linear operators, in this paper we need to restrict ourselves to the right $S$-resolvent.
The reason for this is that, for any $v\in V$, the function $s\mapsto S_R^{-1}(s,T)v$ is left slice regular on $\rho_S(T)$ because $\partial_I S_R^{-1}(s,T)v = ( \partial_I S_R^{-1}(s,T))v = 0$.  Note that although the mapping $s\mapsto S_L^{-1}(s,T)$  is right slice regular on $\rho_S(T)$, the mapping $s\mapsto S_L^{-1}(s,T)v$ is in general not right slice regular because the vector $v$ does not necessarily commute with the operator $\partial_I$.
\end{remark}
\begin{theorem}Let $T\in\closOP(V)$.
\begin{enumerate}[(i)]
\item If $\alpha\in\rho_S(T)$ is real then $S_R^{-1}(\alpha,T) = (\alpha\id - T)^{-1}$.
\item The mapping $s\mapsto S_R^{-1}(s,T)$ is left slice regular on $\rho_S(T)$.
\item The right $S$-resolvent operator satisfies the $S$-resolvent equation
\begin{equation}\label{reseqR}
sS_R^{-1}(s,T)v-S_R^{-1}(s,T)Tv=\id v, \quad v\in \dom (T).
\end{equation}
\end{enumerate}
\end{theorem}
In analogy with Definition~\ref{SlicePow} we define
\[ S_R^{-n}(s,T) = \sum_{k=0}^{n} \binom{n}{k} s^{n-k}(-T)^{k} (T^2-2\Re(s)T + |s|^2\id)^{-n}.\]

\begin{corol}\label{SResDeriv}
For $s\in\rho_S(T)$ we have
\[\partial_s^m S_R^{-1}(s,T) = (-1)^m m!\, S_R^{-(m+1)}(s,T),\quad\text{ for any }n\in\nn.\]
\end{corol}

The definition of the $S$-resolvent operator allows us to define the $S$-functional calculus for bounded quaternionic right linear operators.
\begin{definition}{\rm
Let $T\in\closOP(V)$. An axially symmetric slice domain $U$ is called $T$-admissible if $\sigma_S(T)\subset U$ and if for any $I\in\S$ the boundary $\partial(U\cap\C_I)$ consists of the finite union of continuously differentiable Jordan curves.

We define $\mon^R_{\sigma_S(T)}$ to be the set of all functions that are right slice regular on an open set $O$ such that there exists a $T$-admissible slice domain $U$ whose closure $\overline{U}$ is contained in $O$.}
\end{definition}
\begin{definition}[$S$-functional calculus for bounded operators]{\rm
Let $T\in\boundOP(V)$, let $I\in\S$ and set $ds_I = - ds I$. We define for any $f\in\mon^R_{\sigma_S(T)}$
\begin{equation}\label{SCalcIntB}f[T] := \frac{1}{2\pi} \int_{\partial(U\cap\C_I) }f(s)\,ds_I\, S_R^{-1}(s,T), \end{equation}
where this integral is indepentend of the choice of the imaginary unit $I \in\S$ and of the $T$-admissible slice domain $U$.}
\end{definition}

 We say that a function $f$ is right slice regular at infinity if it is right slice regular on the set ${\{s\in\hh: r<|s|\}}$ for some $r>0$ and the limit $\lim_{s\to\infty} f(s)$ exists in $\hh$. In this case we define
\[f(\infty) := \lim_{s\to\infty}f(s).\]
\begin{definition}{\rm
Let $T\in\closOP(T)$. We denote by $\mon^R_{\sigma_S(T)\cup\{\infty\}}$ the set of all functions $f\in\mon^R_{\sigma_S(T)}$ that are right slice regular at infinity.}
\end{definition}

As in the complex case the functional calculus for unbounded operators is defined using a transformation of the unbounded operator into a bounded one.
For $\alpha\in\rr$ we consider the function $\Phi_\alpha:\hh\cup\{\infty\}\to\hh\cup\{\infty\}$ defined by $\Phi_\alpha(s) = (s-\alpha)^{-1}$ for $s\in\hh\setminus\{\alpha\}$, $\phi(\alpha) = \infty$ and $\phi(\infty) = 0$.
\begin{definition}{\rm
Let $T\in\closOP(V)$ be such that $\rho_S(T)\cap\rr\neq\emptyset$, let $\alpha\in \rho_S(T)\cap\rr$ and set $A = (T-\alpha\id)^{-1}$. For any $f\in \mon^R_{\sigma_S(T)\cup\{\infty\}}$ we define
\[ f[T] := (f\circ\Phi_{\alpha}^{-1})(A).\]}
\end{definition}
This definition is independent of the choice of $\alpha\in\rho_s(T)\cap\rr$. Moreover a integral representation corresponding to the one in \eqref{SCalcIntB} holds true as the next theorem shows.
\begin{theorem}
Let $ T\in\closOP(V)$ with $\rho_s(T)\cap\R\neq\emptyset$ and let $f\in\mon^R_{\sigma_S(T)\cup\{\infty\}}$. If $U$ is a $T$-admissible slice domain such that $f$ is slice regular on an open superset of $\overline{U}$ then
\[f[T] = f(\infty)\id + \frac1{2\pi}\int_{\partial(U\cap\C_I)} f(s)\,ds_I\,S_R^{-1}(s,T),\]
where $ds_I = - ds I $ and the integral does not depend on the choice of the imaginary unit $I\in\S$.
\end{theorem}

The functional calculi defined above are consistent with algebraic operations such as addition and multiplication of functions, multiplications with scalars from the left and composition as far as they are supported by the class $\mon^R_{\sigma_S(T)}$. Note that we denote functions of an operator obtained by the $S$-functional calculus with square brackets in order to distinguish them from those obtained by the calculus we define in this paper.

\subsection{Strongly continuous groups of quaternionic operators}
The $S$-functional calculus is the fundamental tool to develop the theory of strongly continuous semigroups and groups of quaternionic operators, cf. \cite{MR2803786}. A family of bounded right-linear operators $(\U(t))_{ t\geq 0}$ on $V$ is called a strongly continuous quaternionic semigroup if $\,\U(0) = \id$ and $\,\U(t_1 + t_2) = \U(t_1) + \U(t_2)$ for $t_1, t_2 \geq 0$ and if $t\mapsto \U(t)v$ is a continuous function on $[0,\infty)$ for any $v\in V$.
\begin{definition}{\rm
Let $(\U(t))_{ t\geq 0}$ be a strongly continuous quaternionic semigroup. Set
\[ \dom(T) =\left\{ v\in V : \lim_{h\to0^+} \frac{1}{h}(\U(h)v - v)\; {\rm exists} \right\}\]
and
\[Tv = \lim_{h\to0^+} \frac{1}{h}(\U(h)v - v), \quad v\in \mathcal{D}(T).\]
The operator $T$ is called the quaternionic infinitesimal generator of the semigroup $(\U(t))_{t \geq 0}$.

We indicate that $T$ is the infinitesimal generator of the semigroup $(\U(t))_{t \geq 0}$ by writing $\U_T(t)$ instead to $\U(t)$.}
\end{definition}
The set $\dom(T)$ is a right subspace that is dense  in $V$ and $T: \dom(T)\to V$ is a closed quaternionic right linear operator. Moreover
\[ \U_T(t) = \exp[tT],\quad t\geq 0.\]

\begin{theorem}
Let $(\U_T(t))_{t \geq 0}$ be a strongly continuous quaternionic semigroup and let $T$ be its quaternionic infinitesimal generator. Then
\[\omega_0 := \lim_{t\to\infty}\frac{1}{t}\ln\left\|\U_T(t)\right\| <\infty.\]
If $s\in\hh$ with $\Re(s) > \omega_0$  then $s$ belongs $\rho_S(T)$ and
\[ S_R^{-1}(s,T) = \int_{0}^\infty e^{-ts}\:\!\U(t)\,dt. \]
\end{theorem}

The question whether a closed linear operator is the infinitesimal generator of a strongly continuous semigroup is answered by the Hille-Yoshida-Phillips Theorem.

\begin{theorem}
Let $T$ be a closed linear operator with dense domain. Then $T$ is the infinitesimal generator of a strongly continuous semigroup if and only if there exist constants $\omega\in\rr$ and $M>0$ such that $\sigma_S(T) \subset\{s\in\hh: \Re(s) \leq \omega\}$ and such that for any $s_0\in\rr$ with $ s_0>\omega$
\[\left\|(S_R^{-1}(s_0,T))^n\right\|\leq \frac{M}{(s_0 - \omega)^n} \qquad \text{for }n\in\nn.\]
\end{theorem}

We consider the problem to characterize when a strongly continuous semigroup of operators $(\U_T(t))_{t\geq 0 }$
can be extended to a group $(\Z_T(t))_{t\in\rr}$ of operators.
This extension is unique if it exists and if the family $\U_-(t)= \Z_T(-t)$,  $t\geq 0$,
is a strongly continuous semigroup.
Consider the identity
$$
\frac{1}{h}[\U_-(h)v-v]=\frac{1}{-h}[-\Z_T(-2)[ \Z_T(2-h)v-\Z_T(2)v]],\ \ \ {\rm for}\ \ h\in (0,1).
$$
By taking the limit for $h\to 0$ we have that the infinitesimal generator of $\U_{-}(t)$ is $-T$ and $\dom (-T)=\dom (T)$.
In this case $T$ is called the quaternionic infinitesimal generator of the group $(\Z_T(t))_{t \in \rr}$.
The next theorem gives a necessary and sufficient condition such that a semigroup can be extended to a group.
\begin{theorem}\label{GroupGen}
An operator $T\in\closOP(V)$ is the quaternionic infinitesimal generator of a strongly continuous group of bounded quaternionic linear operators if and only if
 there exist real numbers $M>0$ and $\omega \geq 0$ such that
\begin{equation}\label{paleologo}
\|(S_R^{-1}(s_0,T))^n\|\leq \frac{M}{(|s_0|-\omega)^n} ,\quad \text{for } \omega<|s_0|.
\end{equation}
If $T $ generates the group  $(\Z_T(t))_{t\in \rr}$, then $\|\Z_T(t)\|\leq Me^{\omega |t|}$.
\end{theorem}

\section{The main results}\label{sec3}

The functional calculus considered in this paper is based on the quaternionic Laplace-Stieltjes-transform. In order to define it, we recall some well-known results on complex measures in the quaternionic setting.
\subsection{Preliminary of quaternionic measure theory}
Let $(\Omega, \mathcal A)$ be a measurable space. A {\em quaternionic measure} is a function $\mu\colon\mathcal A\to \hh $ that satisfies
$$\mu\left(\mathbb \bigcup_{n\in\mathbb N}A_n\right) = \sum_{n\in \mathbb N} \mu(A_n)$$
for any sequence of pairwise disjoint sets $(A_{n})_{n\in\mathbb N}\subset\mathcal A$. We denote the set of all quaternionic measures on $\mathcal A$ by $\meas(\Omega,\mathcal A, \hh )$ or simply by $\meas(\Omega,\hh)$ or $\meas(\Omega)$ if there is no possibility of confusion.
The set of all quaternionic measures is a two-sided quaternionic vector space with the operations
$(\mu+\nu)(A) = \mu(A) + \nu(A)$, $(a\mu)(A) = a\mu(A)$ and $(\mu a)(A) = \mu(A)a$
for ${\mu,\nu\in \meas(\Omega)}$,  $a\in\hh $ and $A\in\mathcal A$.

\begin{remark}\label{MSplit}
Let $I,J\in\mathbb S$ with $I\perp J$. Since $\hh = \cc_I + J\cc_I$, it is immediate  that a mapping $\mu: \mathcal A\to\hh$ is a quaternionic measure if and only if there exist two $\C_I$-valued complex measures $\mu_1,\mu_2$ such that $\mu(A) = \mu_1(A) + J\mu_2(A)$ for any $A\in\mathcal{A}$.

Moreover, since also $\hh = \cc_I + \cc_I J$, there exist  $\cc_I$-valued measures $\tilde{\mu}_1,\tilde{\mu}_2$ such that $\mu(A) = \tilde{\mu}_1(A) + \tilde{\mu}_2(A) J$ for any $A\in\mathcal{A}$.
\end{remark}
Agrawal and Kulkarni showed in \cite[Section~3]{agrawal} that one can define the variation of a quaternionic measure just as for complex measures
and that  the Radon-Nikod\`ym theorem also holds true in this setting.
\begin{definition}{\rm
Let $\mu\in \meas(\Omega, \mathcal A, \hh )$. For all $A\in\mathcal A$  denote by $\Pi(A)$ the set of all countable partitions $\pi$ of $A$ into pairwise disjoint, measurable sets $A_i, i\in\mathbb N$. We call the set function
\[|\mu|(A) = \sup\Big\{\sum_{A_i\in\pi}|\mu(A_i)|\,\Big|\, \pi\in\Pi(A)\Big\}\qquad\text{for all }A\in\mathcal A,\]
 the {\em total variation} of $\mu$.}
\end{definition}
\begin{corol}\label{VarProp}
The total variatation $|\mu|$ of a measure $\mu\in \meas(\Omega, \mathcal A, \hh )$ is a finite positive measure on $\Omega$.
Moreover,  $|a \mu| = |\mu a | =  |\mu| |a|$ and $|\mu + \nu| \leq |\mu| + |\nu|$ for any $\mu,\nu\in \meas(\Omega, \mathcal A, \hh )$ and $a\in \hh$.
\end{corol}

Recall that a measure $\mu$ is called {\em absolutely continuous} with respect to a positive measure $\nu$, if $\mu(A) = 0$ for any $A\in\mathcal A$ with $\nu(A) = 0$. In this case, we write $\mu\ll\nu$.

\begin{theorem}[Radon-Nikod\`ym theorem for quaternionic measures]\label{RadonNikodym}
Let $\nu$ be a $\sigma$-finite positive measure on $(\Omega, \mathcal A)$. A quaternionic measure $\mu\in \meas(\Omega, \mathcal A,\hh )$ is absolutely continuous with respect to  $\nu$ if and only if there exists a unique function $f\in L^1(\Omega,\mathcal A,\nu,\hh )$ such that
\[\mu (A) = \int_{A} f\, d\nu\qquad\text{for all }A\in\mathcal A.\]
Moreover, in this case
\begin{equation}\label{VarIdentitiy}|\mu| (A) = \int_{A} |f|\, d\nu\qquad\text{for all }A\in\mathcal A.\end{equation}
\end{theorem}
The identity \eqref{VarIdentitiy} follows as in the classical case, cf. \cite[Theorem~6.13]{rudin}.

\begin{corol}\label{ACVariation}
Let $\mu\in \meas(\Omega, \mathcal A,\hh)$. Then there exists an $\mathcal A$-measurable function $f:\Omega\to\hh$ such that $|h(x)| = 1$ for any $x\in\Omega$ and such that $\mu(A)  = \int_Ah \,d|\mu|$ for any $A\in\mathcal A$.
\end{corol}



Let $V$ be a quaternionic Banach space and let $\nu$ be a positive measure. Recall that $V$ is also a real Banach space  if we restrict the scalar multiplication to the real numbers. Moreover, recall that $\hh$ itself is a quaternionic Banach space.

Let $\nu$ be a positive measure. Recall that in Bochner's integration theory, a function $f$ with values in $V$ is called {\em $\nu$-measurable} if there exists a sequence of functions $f_n(x) =\sum_{i=1}^n a_i\indi_{A_i}(x)$, where $a_i\in V$ and $\indi_{A_i}$ is the characteristic function of a measurable set $A_i$, such that $f_n(x) \to f(x)$ as $n\to\infty$ for $\nu$-almost all $x\in\Omega$. Pettis measurability theorem gives an equivalent condition to $\nu$-measurability.

\begin{theorem}Let $V$ be a real Banach space and let $\nu$ be a positive measure. A function with values in $X$ is $\nu$-measurable if and only if
\begin{enumerate}[(i)]
\item there exist a $\nu$-zero set $N$ and a separable subspace $B$ of $V$ such that $f(\Omega\setminus N)\subset B$ and
\item $x\mapsto \langle v^\ast,f(x)\rangle$ is measurable for any $v^\ast$ in the topological dual space of the real Banach space $X$.
\end{enumerate}
\end{theorem}

\begin{lemma}
Let $X$ be a quaternionic Banach space and let $\nu$ be a positive measure on $(\Omega,\mathcal A)$. If $f:\Omega\to X$ and $g:\Omega\to\hh$ are $\nu$-measurable, then the functions $fg$ and $gf$ are $\nu$-measurable.
\end{lemma}
\begin{proof}
Since $f$ is $\nu$-measurable, there exist a set $N\subset\Omega$ with $\nu(N) = 0$ and a separable $\rr$-linear subspace $B$ of $X$ such that $f(\Omega\setminus N) \subset B$. Let $g = g_0 + \sum_{i=1}^3 g_i e_i$ where the $g_i:\Omega \to \rr$ are the real-valued component functions of $g$. Since $g$ is $\mu$-measurable, its components are measurable. For any continuous $\rr$-linear functional $v^\ast: X\to \rr$, we have
\[\langle v^\ast, fg\rangle = \langle v^\ast, f\rangle g_0  + \langle v^\ast, fe_1\rangle g_1  + \langle v^\ast, fe_2\rangle g_2  + \langle v^\ast, fe_3\rangle g_3.\]

If $f_n$ is a sequence of $\mathcal A$-simple functions such that $f_n(x)\to f(x)$ for $\nu$-almost all ${x\in\Omega}$, then $f_n(x)e_i$ is a sequence of $\mathcal A$-simple functions with $f_n(x)e_i \to f(x)e_i$ for $\nu$-almost all ${x\in\Omega}$. Therefore, each of the mappings $x\mapsto f(x)e_i,i=1,\ldots,3$ is $\mu$-measurable. The Pettis measurability implies that the mapping $x\mapsto \langle v^\ast, f(x)\rangle$ and the mappings $x\mapsto \langle v^\ast, f(x)e_i\rangle$ for  $i=1,\ldots,3$ are measurable. Consequently, the mapping $x\mapsto \langle v^\ast, f(x)g(x)\rangle$ is measurable, since it consists of the products and the sum of measurable functions. Moreover, $f(x)g(x)\in B + Be_1 + Be_2 + B e_3$ for any $x\in\Omega\setminus N$. This space is separable because $B$ is separable.  Since $\nu(N) = 0$ and since $v^\ast$ was arbitrary, the mapping $x\mapsto f(x)g(x)$ is $\mu$-measurable by the Pettis measurability theorem.

Similarly, we obtain that  $\nu$-measurability of $gf$.

\end{proof}
Let $\nu$ be a positive measure on a $(\Omega,\mathcal A)$. Recall that a $\nu$-measurable function on $\Omega$ with values in a real Banach space is Bochner-integrable, if and only if $\int_{\Omega} \| f\| d\mu <\infty$.
\begin{definition}{\rm
Let $V$ be a quaternionic Banach space, let $\mu\in \meas(\Omega,\mathcal A,\hh)$ and let $h:\Omega\to\hh$ be the function with $|h| = 1$  such that $\mu = h\,d|\mu|$. We call two $\mu$-measurable functions $f:\Omega\to X$  and $g:\Omega\to \hh$ a {\em $\mu$-integrable pair}, if
\[\int_{\Omega} \|f\|\|g\|\,d|\mu| <\infty.\]
In this case, we define
 \begin{equation}\label{IntDefi}\int_{\Omega}f\,d\mu\, g = \int_{\Omega}fhg\,d|\mu| \quad\text{and}\quad \int_{\Omega}g\,d\mu\, f = \int_{\Omega}ghf\,d|\mu|,\end{equation}
as the integrals of a function with values in a real Banach space in the sense of Bochner.}
\end{definition}

\begin{remark}
Note that in the definition of the integral in \eqref{IntDefi}, we can replace the variation of $\mu$ by any $\sigma$-finite positive measure $\nu$ with $\mu\ll\nu$. If $h_\nu$ is the density of $\mu$  with respect to $\nu$ and  $\rho_{|\mu|}$ and $\rho_\nu$ are the real-valued densities of $|\mu|$ and $\nu$ with respect to $|\mu| +\nu$. Then \[\mu = h |\mu| = h\rho_{|\mu|} (|\mu| + \nu) \quad\text{and}\quad\mu = h_\nu \nu = h_\nu\rho_{\nu} (|\mu| + \nu).\]
Theorem~\ref{RadonNikodym} implies $h\rho_{|\mu |} = h_{\nu}\rho_\nu$ in $L^1(|\mu| + \nu)$. Therefore
\begin{align*}
\int_{\Omega} f h_{\nu}g\, d\nu &= \int_{\Omega} \int_{\Omega} f h_{\nu}g \rho_{\nu}\, d(|\mu|+\nu) =  \int_{\Omega} \int_{\Omega} f h_{\nu} \rho_{\nu}g\, d(|\mu|+\nu)\\
&= \int_{\Omega} f h\rho_{|\mu|} g \,d(|\mu|+\nu) =  \int_{\Omega} f h g \rho_{|\mu|}\, d(|\mu|+\nu) = \int_{\Omega} f h g\, d|\mu| = \int_{\Omega} f\,d\mu\, g.
\end{align*}
Hence, the integral is linear in the measure: if $\mu,\nu\in\meas(\Omega,\mathcal{A},\hh)$ then $\mu$ and $\nu$ are absolutely continuous with respect to $\tau = |\mu|+|\nu|$. If $\rho_\mu$ and $\rho_\nu$ are the densities of $\mu$ and $\nu$ with respect to $\tau$, then
\[\int_{\Omega}f\, d(\mu + \nu)\, g = \int_{\Omega} f(\rho_\mu + \rho_\nu) g\, d\tau = \int_{\Omega} f \rho_\mu g \,d\tau + \int_{\Omega} f\rho_\nu g\, d\tau =  \int_{\Omega} f\,d\mu\, g + \int_{\Omega} f\,d\nu\, g .\]
Similarly, if $a\in\hh$ and $\mu =  \rho |\mu|$ then $a\mu  = a \rho  |\mu|$  and so
\[\int_{\Omega} f\, d(a\mu )\, g = \int_{\Omega} f (a\rho)  g\, d|\mu| = \int_{\Omega} (f a)\rho  g\, d|\mu| = \int_{\Omega} (fa)\,d\mu\, g.\]
In the same way, one can see that \(\int_{\Omega}f\,d(\mu a)g = \int_{\Omega}f\, d\mu\, (ag)\).
\end{remark}
As for complex measures, it is possible to define the product measure of two quaternionic measures.
\begin{lemma}\label{LemProdMeas}
Let $\mu\in \meas(\Omega, \mathcal A,\hh )$ and $\nu\in \meas(\varUpsilon,\mathcal B, \hh )$. Then there exists a unique measure $\mu\times\nu$  on the product measurable space $(\Omega\times\varUpsilon,\mathcal A\otimes\mathcal B)$ such that
\begin{equation}\label{prodMeas}\mu\times\nu(A\times B) = \mu(A)\nu(B)\end{equation}
for all $A\in\mathcal A, B\in\mathcal B$. We call $\mu\times\nu$  the {\em product measure} of $\mu$ and $\nu$.
\end{lemma}
\begin{proof}
Let $I,J\in\mathbb S$ with $I\perp J$ and let $\mu = \mu_1 + J\mu_2 $ with $\mu_1,\mu_2\in M (\Omega, \mathcal A, \cc _I)$ and  $\nu = \nu_1 + \nu_2 J$ with $\nu_1,\nu_2\in \meas(\varUpsilon,\mathcal B, \cc _I)$. Then, there exist unique complex product measures $\mu_i\times\nu_j\in M (\Omega_1\times\Omega_2,\mathcal A\otimes\mathcal B,\cc _I)$ of $\mu_i$ and $\nu_i$, $i,j =1,2$. If we set
$$\mu\times\nu = \mu_1\times\nu_1 + J\mu_2\times\nu_1 + \mu_1\times\nu_2J + J\mu_2\times\nu_2 J,$$
then $\mu\times\nu$ is a quaternionic measure on $(\Omega\times\varUpsilon,\mathcal A\otimes\mathcal B)$ and
\begin{align*}
\mu(A)\nu(B) &= \mu_1(A)\nu_1(B) + J\mu_2(A)\nu_1(B) + \mu_1(A)\nu_2(B)J + J \mu_2(A)\nu_2(B)J&\\
 &= \mu_1\times\nu_1(A\times B) + J\mu_2\times\nu_1(A\times B) + \mu_1\times\nu_2(A\times B)J + J\mu_2\times\nu_2(A\times B) J \\
 &= \mu\times\nu(A\times B).
\end{align*}

In order to prove the uniqueness of the product measure, assume that two quaternionic measures $\rho = \rho_1 + \rho_2J$ and $\tau = \tau_1 + \tau_2 J$ on $(\Omega\times \varUpsilon, \mathcal A\times \mathcal B)$ satisfy $\rho(A\times B)=\tau(A\times B)$ whenever $A\in\mathcal A$ and $B\in\mathcal B$. Then $\rho_1(A\times B) =\tau_1(A\times B)$ and $\rho_2(A\times B) = \tau_2(A\times B)$ for $A\in \mathcal A$ and $B\in\mathcal B$. Since two complex measures on the product space $(\Omega\times\varUpsilon,\mathcal A\otimes\mathcal B)$ are equal if and only if they coincide on sets of the form $A\times B$, we obtain $\rho_1 = \tau_1$ and $\rho_2 = \tau_2$ and, in turn, $\rho = \rho_1 + \rho_2 J = \tau_1 + \tau_2 J = \tau$. Therefore, $\mu\times\nu$ is uniquely determined by \eqref{prodMeas}.

\end{proof}

\begin{remark}
Note that it is also possible to define a commutative product measure $\mu\times_c\nu$ that satisfies
\[\mu\times_c\nu(A\times B) = \nu(B)\mu(A),\quad \forall A\in\mathcal A, B\in\mathcal B.\]
This measure is different from the measure $\nu\times\mu$ that satisfies
\[\nu\times\mu(B\times A) = \nu(B)\mu(A),\quad \forall B\in\mathcal B,  A\in\mathcal A.\]
\end{remark}

\begin{lemma}\label{ProdVar}
Let $(\Omega, \mathcal A, \mu)$ and $(\varUpsilon, \mathcal B, \nu)$ be quaternionic measure spaces. Then$$|\mu\times\nu| = |\mu|\times |\nu|.$$ Moreover, if $\mu = f\,d|\mu|$ and $\nu = g\,d|\nu|$ as in Corollary~\ref{ACVariation}, then for any $C\in\mathcal A\times\mathcal B$
\[\mu\times\nu(C) = \int_{C}f(s)g(t)\, d|\mu\times \nu|(s,t).\]
\end{lemma}
\begin{proof}
Let $f\colon\Omega\to\hh $ and  $g\colon \varUpsilon\to\hh $ with $|f| = 1$ and $|g|=1$ be functions as in Corollay~\ref{ACVariation}  such that $\mu(A) = \int_A f(t)\,d|\mu|(t)$ and $\nu(B) = \int_B g(s)\, d|\nu|(s)$ for all $A\in\mathcal A$ and $ B\in\mathcal B$. Moreover, let $r=(t,s)$ and let $h(r) = f(t)g(s)$. Then $C \mapsto \int_C h(r)\,d|\mu|\times|\nu|(r)$ defines a measure on $(\Omega\times\varUpsilon, \mathcal A\times\mathcal B)$ and Fubini's theorem for positive measures implies
\begin{align*}\int_{A\times B}h(r)\,d|\mu|\times|\nu|(r) &= \int_{A}\int_Bf(t)g(s)\,d|\mu|(t)\,d|\nu|(s) \\
&= \int_{A}f(t)\,d|\mu|(t)\int_Bg(s)\,d|\nu|(s) = \mu(A)\nu(B).\end{align*}
The uniqueness of the product measure implies
$\mu\times\nu (C) = \int_C h(r) \,d \big|\mu|\times|\nu|(r)$ for any $C\in\mathcal A\times\mathcal B$. Since $|h| = |f|\,|g| = 1$, we deduce from \eqref{VarIdentitiy} that
$$ |\mu\times\nu|(C) = \int_C |h|\,d|\mu|\times|\nu|(r) = |\mu|\times|\nu|(C)$$
for all $C\in \mathcal{A}\times\mathcal B$.
\end{proof}

\begin{lemma}\label{LemImagInt}
Let $(\Omega,\mathcal A,\mu)$ be a quaternionic measure space, let $(\varUpsilon,\mathcal B)$ be a measurable space and let $\phi: \Omega\to\varUpsilon$ be a measurable function. If a function $f:\varUpsilon\to X$ with values in a quaternionic Banach space $X$ is integrable with respect to $\mu^\phi$ and $f\circ\phi$ is integrable with respect to $\mu$, then
\begin{equation} \label{ImagInt}\int_{\varUpsilon} f \,d\mu^{\phi} = \int_{\Omega}f\circ\phi\, d\mu.\end{equation}
\end{lemma}
\begin{proof}
Let $I,J\in\si$ such that $I\perp J$ and let $\mu_1,\mu_2\in \meas(\Omega,\mathcal A,\cc_I)$ such that $\mu = \mu_1 + \mu_2J$. Then $\mu^\phi = \mu_1^\phi + \mu_2^\phi J$ and $|\mu_i|\leq|\mu|, i=1,2$, which implies $\int_{\Omega}\|f\circ\phi\| d|\mu_i|\leq \int_{\Omega}\|f\circ\phi\| d|\mu| <\infty$. Hence, $f\circ \phi$ is integrable with respect to $\mu_1$ and $\mu_2$ and similarly $f$ is integrable with respect to $\mu_1^\phi$ and $\mu_2^\phi$. Therefore, \eqref{ImagInt} holds true for the complex measures $\mu_1$ and $\mu_2$, and hence
\[ \int_{\varUpsilon} f \,d\mu^{\phi} = \int_{\varUpsilon} f \,d\mu_1^{\phi} + \int_{\varUpsilon} f \,d\mu_2^{\phi} J = \int_{\Omega}f\circ\phi\, d\mu_1 + \int_{\Omega}f\circ\phi\, d\mu_2 J= \int_{\Omega}f\circ\phi\, d\mu.\]

\end{proof}

\begin{definition}{\rm
We denote the Borel sets on $\rr$ by $\borel(\rr)$.}
\end{definition}
We recall that, for any Borel set $E\subset\rr$, the set
\[
P(E):=\{(u,v)\in\rr^2 : u+v\in E \}
\]
is a Borel subset of $\rr^2$.

\begin{definition}\label{DefiConvol}{\rm
Let $\mu,\nu$ be quaternionic measures on $ \mathbf{B}(\rr )$. The convolution $\mu\ast\nu$ of $\mu$ and $\nu$ is the image measure of $\mu\times\nu$ under the mapping $\phi:\rr^2\to\rr, (u,v)\mapsto u+v$, that is,
 \[\mu\ast\nu(E)=\mu\times\nu(P(E))\]
for any $E\in\mathbf{B}(\rr )$.}
\end{definition}
It is immediate to show the following rules.
\begin{corol}
Let $\mu,\nu,\rho\in \meas(\rr , \mathbf{B}(\rr ), \hh )$ and let $a,b\in\hh$. Then
\begin{enumerate}
\item $(\mu + \nu)\ast\rho = \mu\ast \rho + \nu\ast\rho$ and $\mu\ast(\nu+\rho) = \mu\ast\nu + \mu\ast\rho$
\item $(a\mu)\ast\nu = a(\mu\ast\nu)$ and $\mu\ast(\nu a) = \mu\ast\nu a$.
\end{enumerate}
\end{corol}
\begin{corol}\label{VarConvol}
Let $\mu,\nu\in\mathbb \meas(\rr , \mathbf{B}(\rr ), \hh )$. Then the estimate
\[ |\mu\ast\nu|(E) \leq |\mu|\ast |\nu|(E) \]
holds true for all $E\in\mathbf{B}(\rr )$.
\end{corol}
\begin{proof}
Let $E\in\borel(\rr)$ and let $\pi \in\Pi(E)$ be a countable measurable partition of $E$. Then
$$\sum_{E_i\in\pi} |\mu\ast\nu(E_i)| = \sum_{E_i\in\pi} |\mu\times\nu(P(E_i))| \leq \sum_{E_i\in\pi} |\mu\times\nu|(P(E_i))= |\mu\times\nu|(P(E)),$$
and taking the supremum over all possible partitions $\pi\in\Pi(E)$ yields
$$|\mu\ast\nu|(E) \leq |\mu\times\nu|(P(E)) = |\mu|\times |\nu|(P(E)) = |\mu|\ast |\nu|(E).$$
\end{proof}

\begin{corol}\label{ConvolSplit}
Let $\mu,\nu\in \meas(\rr , \mathbf{B}(\rr ), \hh )$ and let $F\colon\rr \to X$ be integrable with respect to $\mu\ast\nu$ and such that $\int_{-\infty}^{+\infty}\int_{-\infty}^{+\infty}\|F(s+t)\|\,d|\mu|(s)\,d|\nu|(t)<\infty$. Then
$$\int_{\rr } F(r)\,d(\mu\ast\nu)(r) = \int_{\rr }\int_{\rr }F(s+t)\,d\mu(s)\,d\nu(t).$$
\end{corol}
\begin{proof}
Because of our assumptions and Definition~\ref{DefiConvol} we can apply Lemma~\ref{LemImagInt} with $\phi(s,t) = s+t$. If $\mu(A) = \int_A f(t)\,d|\mu|(t)$ and $\nu(A) = \int_A g(s)\,d|\nu|(s)$, then the product measures satisfies $\mu\times\nu (B)= \int_B f(s)g(t)\,d|\mu|\times|\nu|(s,t)$ by Lemma~\ref{ProdVar}. Applying Fubini's theorem, we obtain
\begin{align*}
\int_{\rr } F(r)\,d(\mu\ast\nu)(r) &= \int_{\rr } F(\phi(s,t))\ d(\mu\times\nu)(s,t) \\
&= \int_{\rr } F(\phi(s,t)) f(s)g(t)\ d|\mu\times\nu|(s,t) \\
&= \int_{\rr } F(s+t) f(s)g(t)\ d|\mu|(s)\, d|\nu|(t) =\int_{\rr }\int_{\rr }F(s+t)\,d\mu(s)\,d\nu(t).
\end{align*}

\end{proof}

\subsection{The quaternionic Laplace-Stieltjes transform and functions of the generator of a strongly continuous group}
Let $(\Z_T(t))_{t \in \rr}$ be a strongly continuous group of operators on $X$. By Theorem~\ref{GroupGen},  there exist positive constants $M>0$ and $\omega\geq 0$ such that $\|\Z_T(t)\|\leq Me^{\omega |t|}$ and such that the $S$-spectrum of the infinitesimal generator $T$ lies in the strip $-\omega < \Re(s)< \omega$.

Moreover,
\[
S_R^{-1}(s,T)v=\int_0^\infty e^{-ts}\, \Z_T(t) v \,dt,\qquad\Re(s) >\omega
\]
and
\[
S_R^{-1}(s,T)v=-\int_{-\infty}^0 e^{-ts}\,\Z_T(t)v\,dt, \qquad\Re(s) <-\omega.
\]

\begin{definition}{\rm
We denote by $\mathbf{S}(T)$ the family of all quaternionic measures $\mu$  on $\borel(\rr)$ such that
$$ \int_{\rr }d|\mu|(t)\,e^{(\omega+\varepsilon)|t|} <\infty\ $$
for some $\varepsilon =\varepsilon(\mu)>0$. The function
\[
\lap(\mu)(s) =\int_{\rr }d\mu(t)\,e^{-st} ,\qquad -(\omega+\varepsilon)<\Re(s) <(\omega+\varepsilon)
\]
is called the {\em quaternionic bilateral (right) Laplace-Stieltjes transform} of $\mu$.

We denote by
$\mathbf{V}(T)$ the set of quaternionic bilateral Laplace-Stieltjes transforms of measures in $ \mathbf{S}(T)$.}
\end{definition}
\begin{lemma}\label{ConProd}
Let $\mu,\nu\in \mathbf{S}(T)$ and $a\in\hh$.
\begin{enumerate}[(i)]
\item The measures $a \mu$ and $\mu + \nu$ belong to $\mathbf{S}(T)$ and $\lap(a\mu ) = a\lap(\mu)$ and $\lap(\mu + \nu) = \lap(\mu) + \lap(\nu)$.
\item The measures $\mu\ast\nu$ belongs to $ \mathbf{S}(T)$. If $\nu$ is real-valued, then $\lap(\mu\ast\nu)=\lap(\mu)\ast\lap(\nu)$.
\end{enumerate}
\end{lemma}
\begin{proof}
Let $\varepsilon = \min\{\varepsilon(\mu), \varepsilon(\nu)\}$. Corollary~\ref{VarProp} implies
\[\int_{\rr} d|a\mu |\, e^{|t|(\omega + \varepsilon)} = |a|\int_{\rr} d|\mu|\,e^{|t|(\omega + \varepsilon)} <\infty\]
and
\[\int_{\rr} d|\mu +\nu|\,e^{|t|(\omega + \varepsilon)} \leq \int_{\rr}d|\mu|\,e^{|t|(\omega + \varepsilon)} + \int_{\rr}d|\nu|\,e^{|t|(\omega + \varepsilon)}< \infty.\]
Thus, $a\mu$ and $\mu + \nu$ belong to $\mathbf{S}(T)$. The relations $\lap(a\mu) = a \lap(\mu)$ and $\lap(\mu + \nu)=\lap(\mu) + \lap(\nu)$ follow from the left linearity of the integral in the measure.

The variation of the convolution of $\mu$ and $\nu$ satisfies $|\mu\ast\nu|(E) \leq |\mu|\ast|\nu|(E)$ for any Borel set $E\in\borel(\rr)$, cf. Corollary~\ref{VarConvol}. In view of Corollary \ref{ConvolSplit}, we have
\begin{align*}\int_{\rr }d|\mu\ast\nu|(r) e^{(w+\varepsilon)|r|} &\leq \int_{\rr }\int_{\rr }d|\mu|(s)\, d|\nu|(t) e^{(w+\varepsilon)|s+t|}\\
&\leq \int_{\rr } d|\mu|(s)\,e^{(w+\varepsilon)|s|}\int_{\rr } d|\nu|(t)\,e^{(w+\varepsilon)|t|} <\infty.
\end{align*}
Therefore, $\mu\ast\nu\in \mathbf{S}(T)$. If $\nu$ is real-valued, then $\nu$ commutes with $e^{-s t}$ and
Fubini's theorem implies for $s\in\hh$ with $-(\omega+\varepsilon)<\Re(s) <\omega + \varepsilon$
\begin{align*}\lap(\mu\ast\nu)(s) &= \int_{\rr }d(\mu\ast\nu)(r)\, e^{-s r} = \int_{\rr }\int_{\rr } d\mu(t)\, d\nu(u)\,e^{-s (t+u)}\\
&= \int_{\rr } d\mu(t)\,e^{-s t}\,\int_{\rr } d\nu(u)\,e^{-s u} = \lap(\mu)(s)\,\lap(\nu)(s).
\end{align*}

\end{proof}

\begin{theorem}\label{p.642Lemma2}
Let $f\in \mathbf{V}(T)$ with \(f(s)=\int_{\rr }d\mu(t)\,e^{-st}\) for $-(\omega+\varepsilon)<\Re(s) <\omega+\varepsilon$.

\begin{enumerate}[(i)]
\item The function $f$ is right slice regular on the strip $\{ s\in \mathbb{H}\  :\ -(\omega+\varepsilon)<\Re(s) <\omega+\varepsilon \}$.
\item For any $n\in\mathbb N$, the measure $\mu^{n}$ defined by
$$
\mu^{n}(E)=\int_E d\mu(t)\,(-t)^n\qquad \text{for }E\in\borel(\rr )$$
belongs to $\mathbf{V}(T)$ and, for $s$ with $-(\omega+\varepsilon)<\Re(s) <\omega+\varepsilon$, we have
\begin{equation}
\label{MeasDeriv}\partial_s^nf(s)=\int_{\rr }d\mu^{n}(t)\,e^{-st}=\int_{\rr }  d\mu(t)\, (-t)^n e^{-st},
\end{equation}
where $\partial_s f$ denotes the slice derivative of $f$ as in Definition~\ref{SliceDeriv}.
\end{enumerate}
\end{theorem}
\begin{proof}
For every $n\in \mathbb{N}$ and every $0<\varepsilon_1 <\varepsilon$ there exists a constant $K$ such that
$$
|t|^ne^{(\omega+\varepsilon_1)|t|}\leq K e^{(\omega+\varepsilon)|t|}, \ \ t\in \rr .
$$
Since $\mu\in \mathbf{S}(T)$, we have
\[
\int_{\rr } d|\mu^{n}|(t)\,e^{(\omega +\varepsilon_1)|t|}=
\int_{\rr } d|\mu|(t)\,|t|^ne^{(\omega +\varepsilon_1)|t|}
\leq K\int_{\rr } d|\mu|(t)\,e^{(\omega +\varepsilon)|t|}<\infty
\]
and so  $\mu^{n}\in \mathbf{S}(T)$.

Let $I,J\in \si$ with $I\perp J$ and let $f_I = f_1 + Jf_2$ where $f_1$ and $f_2$ have values in $\cc_I$. Then $ f_I\partial_I  = 0$ if and only if $f_1$ and $f_2$ are holomorphic, which is equivalent to the existence of the limit
\[f_1'(s) + Jf_2'(s) = \lim_{\C_I\ni p\to s} \ \frac{f_1(p) - f_1(s)}{p-s} + J\frac{f_2(p)-f_2(s)}{p-s} =\lim_{\C_I\ni p\to s}\big(f_I(p)-f_I(s)\big)(p-s)^{-1}.\]
For any $s = s_0 + Is_1$ with $-(\omega+\varepsilon)<\Re(s)<\omega+\varepsilon$, we have
\[ \lim_{\C_I\ni p\to s}(f_I(p)-f_I(s))(p-s)^{-1}=\lim_{\C_I\ni p\to s}\int_{\rr }d\mu(t)\,\frac{e^{-pt}-e^{-st}}{p-s}.\]
If $p$ is sufficiently close to $s$ such that also $-(\omega+\varepsilon) < \Re(p)<\omega+\varepsilon$, then the simple calculation
\[
|e^{-p t} - e^{-s t}| = \left|\int_{0}^1e^{-ts -t\xi(p-s)}t(p-s)\, d\xi\right| \leq |t| e^{(\omega+\varepsilon)|t|} |p-s|,
\]
yields the estimate
\[
\frac{|e^{-pt}-e^{-st}|}{|p-s|}\leq |t|e^{(\omega+\varepsilon)|t|},
\]
which allows us to apply Lebesgue's theorem of dominated convergence in order to exchange limit and integration. We obtain
\begin{equation}\label{limitDs} \lim_{p\in\C_I\to s}(f_I(p)-f_I(s))(p-s)^{-1}=\int_{\rr } d\mu(t)\,(-t)e^{-st}=\int_{\rr } d\mu^{1}(t)\, e^{-st}.\end{equation}
Consequently, $f$ is right slice regular on the strip $\{s\in\hh: -(\omega + \varepsilon)<\Re(s)< \omega + \varepsilon \}$.
Moreover, \eqref{limitDs} implies
\[
\partial_sf(s)=\int_{\rr } d\mu^{1}(t)\,e^{-st}
\]
for $-(\omega + \varepsilon)< \Re(s)<\omega+\varepsilon$. By induction we get \eqref{MeasDeriv}.

\end{proof}

\begin{definition}[Functions of the quaternionic infinitesimal generator]{\rm
Let $T$ be the quaternionic infinitesimal generator of the strongly continuous group $(\Z_T(t))_{t \in \rr}$ on a quaternionic Banach space $V$. For $f\in \mathbf{V}(T)$ with
\[
f(s)=\int_{\rr }d\mu(t)\,e^{-st }  \quad\text{for } -(\omega+\varepsilon)<\Re(s)< \omega+\varepsilon,
\]
where $\mu\in \mathbf{S}(T)$, we define the right linear operator $f(T)$ on $V$ by
\begin{equation}\label{CalcDefi}
f(T)v=\int_{\rr } d\mu(t)\, \Z_T(-t)v\qquad\text{for }v\in V.
\end{equation}}
\end{definition}
\begin{remark}\label{SRCompatible}
Note that in particular for $p\in\hh$ with $\Re(p)<-\omega$ the function $s\mapsto S_R^{-1}(p,s)$ lies in $\mathbf{S}(T)$.  Set $\mu_p = - \indi_{[0,\infty)}(t) e^{tp} \, dt$, where $\indi_A$ denotes the characteristic function of a set $A$. If $\Re(p)<\Re(s)$, then
\begin{equation*}
\lap(\mu_p)(s) = \int_{\rr}d\mu_p(t)\,e^{-ts} = -\int_{0}^{\infty} e^{tp}e^{-ts}\,dt = - S_L^{-1}(s,p) = S_R^{-1}(p,s)
\end{equation*}
and
\begin{align*}\lap(\mu_p)(T) &= \int_{\rr}d\mu_p(t)\,\Z(-t) = -\int_{0}^{\infty}e^{tp}\,\Z(-t)\,dt = - \int_{-\infty}^{0}e^{-tp}\,\Z(t)\,dt =  S_R^{-1}(p,T).
\end{align*}

For $p\in\hh$ with $\omega < \Re(p)$ set $\mu_p  = \chi_{(-\infty,0]}(t) e^{tp} \, dt$. Similar computations show that also in this case $S_R^{-1}(p,s)=\lap(\mu_p)(s) \in\mathbf{S}(T)$ if $\Re(s)<\Re(p)$ and $\lap(\mu_p)(T) = S_R^{-1}(p,T)$.
\end{remark}
\begin{theorem}
For any $f\in\mathbf{V}(T)$, the  operator $f(T)$ is bounded.
\end{theorem}
\begin{proof}
Let $f(s) = \int_{\rr}d\mu(t)\,e^{-st}\in\mathbf{V}(T)$ with $\mu\in\mathbf{S}(T)$. Since $\mathcal{\|U}_T(t)\| \leq Me^{w|t|}$,  we have
$$\| f(T) v\| \leq \int_{\rr }d|\mu|(t)\, \|\Z_T(-t)\|\,\|v\| \leq M\int_{\rr }d|\mu|(t)\,e^{w|t|}\|v\|.$$
for any $v\in V$. Thus, $f(T)$ is bounded.

\end{proof}

\begin{lemma}\label{PFCLT}
Let $f = \lap(\mu)$ and $g= \lap(\nu)$ belong to $\mathbf{V}(T)$ and let $a\in\hh$.
\begin{enumerate}[(i)]
\item $(af)(T)=a f(T)$ and $(f+g)(T)=f(T)+g(T)$.
\item If $g$ is an intrinsic function, then $\nu$ is real valued and $(fg)(T)=f(T)g(T)$.
\end{enumerate}
\end{lemma}
\begin{proof}
The statement (i) follows immediately from Lemma~\ref{ConProd} and the left linearity of the integral \eqref{CalcDefi} in the measure.

Consider (ii) and write for $s = s_0 + I s_1$
\[g(s) = \int_{\rr }d\nu(t)\,e^{-st } = \int_{\rr }d\nu(t)\,e^{-s_0t }cos(-s_1t) + \int_{\rr }d\nu(t)\,e^{-s_0t }\sin(-s_1t)I.\]
If $g$ is intrinsic, then $\alpha(s_0,s_1) = \int_{\rr }d\nu(t)\,e^{-s_0t }\cos(-s_1t)$ and $\beta(s_0,s_1) = \int_{\rr }d\nu(t)\,e^{-s_0t }\sin(-s_1t)$ which implies that also $\nu$ is real-valued, since the Laplace-Stieltjes transform is injective on the set of complex measures. Lemma~\ref{ConProd} gives $fg = \lap(\mu\ast\nu)\in \mathbf{V}(T)$ and
\begin{gather*} (fg)(T)v = \int_{\rr}d(\mu\ast\nu)(r)\, \Z_T(-r)v = \int_\rr\int_\rr d\mu(s)\, d\nu(t)\, \Z_T(-(s+t))v\\
=\int_\rr d\mu(s)\,\Z_T(-s)\int_\rr d\nu(t)\, \Z_T(-t)v= f(T)g(T)v,
\end{gather*}
where we use that $\Z_T(-s)$ and $\nu$ commute because $\nu$ is real-valued.
\end{proof}

\section{Comparison with the $S$-functional calculus}

A natural question that arises is the relation between the functional calculus introduced in this paper and the $S$-functional calculus for unbounded operators.
In this section we will show that that in the case the function $f$ is slice regular also at infinity  the two functional calculi coincide. In order to prove this, we need specialized versions of Cauchy's integral theorem and the Residue theorem that fit into our setting. The next theorem is analogue to Lemma~4.5.1 in \cite{MR2752913}.
\begin{theorem}\label{CauchyIntAp}
Let $V$ be a two-sided quaternionic Banach space, let $U\subset\hh$ be an axially symmetric open set such that $\partial(U\cap\C_I)$ is a finite number of continuously differentiable Jordan curves and let $O$ be an open set with $\overline{U}\subset O$. If $f:O\to \hh$ is right slice regular and $g: O\to V$ is left slice regular, then, for any $I\in\S$, it holds
\begin{equation*}
\int_{\partial (U\cap\C_I)} f(s)\, ds_I\, g(s) = 0.
\end{equation*}
\end{theorem}


\begin{lemma}\label{ResMod}
Let $O\subset\hh$ be open, let $f: O\setminus[p] \to \hh$ be right slice regular and let $g: O\to V $ be left slice regular such that $ p =p_0+Ip_1\in O$ is a pole of order $n_f\geq 0$ of $f_I$. If $\varepsilon>0$ is such  that $\overline{U_{\varepsilon}(p)\cap\cc_I}\subset O$, then
\[ \frac{1}{2\pi} \int_{\partial (U_{\varepsilon}(p)\cap\C_I)} f(s)\,ds_I\, g(s) = \sum_{k=0}^{ n_f -1} \frac{1}{k!}\Res_{p}\left(f_I(s)(s-p)^{k}\right)\left(\partial_s^k g(p)\right).\]
\end{lemma}
\begin{proof}
Since $f$ is right slice regular, its restriction $f_I$ is a vector-valued holomorphic function on $\C_I$ if we consider $\hh$ as a vector space over $\C_I$ by restricting the multiplication with quaternions on the right  to $\C_I$. Similarly, since $g$ is left slice regular, its restriction $g$ is a $V$-valued holomorphic function if we consider $V$ as a complex vector space over $\C_I$ by restricting the left scalar multiplication to $\C_I$. Consequently, if we set $\rho = \mathrm{dist}(p, \partial(O\cap\C_I)$, then
\begin{equation}\label{SeriesCoef} f_I(s) = \sum_{k=-n_f}^{\infty} a_k (s-p)^k \quad\text{and}\quad g_I(s) = \sum_{k=0}^{\infty} (s-p)^kb_k \qquad\text{for }s\in (U_{\rho}(p)\cap\C_I)\setminus\{p\} \end{equation}
with $a_k\in\hh$ and $b_k\in V$. These series converge uniformly on ${\partial( U_{\varepsilon}(p)\cap\C_I)}$ for any $0<\varepsilon <\rho$. Thus,
\begin{align*}
\frac1{2\pi}\int_{\partial(U_{\varepsilon}(p)\cap\C_I)} f(s)\,ds_I\,g(s) &= \frac1{2\pi}\int_{\partial(U_{\varepsilon}(p)\cap\C_I)}\left( \sum_{k=0}^\infty a_{k-n_f} (s-p)^{k-n_f}\right)\,ds_I\,\left(\sum_{j=0}^{\infty}(s-p)^{j} b_{j}\right) \\
&= \sum_{k=0}^{\infty}\sum_{j=0}^{k} a_{k-j-n_f}\left(\frac1{2\pi}\int_{\partial(U_{\varepsilon}(p)\cap\C_I)} (s-p)^{k-j-n_f}\, ds_I\, (s-p)^{j}\right) b_{j} \\
&= \sum_{k=0}^{\infty}\sum_{j=0}^{k} a_{k-j-n_f}\left(\frac1{2\pi I}\int_{\partial(U_{\varepsilon}(p)\cap\C_I)} (s-p)^{k-n_f}\, ds\right) b_{j}.
\end{align*}
As $\frac1{2\pi I}\int_{\partial(U_{\varepsilon}(p)\cap\C_I)} (s-p)^{k-n_f}\, ds$ equals $1$ if $ k-n_f = -1$ and $0$ otherwise, we obtain
\[\frac1{2\pi}\int_{\partial(U_{\varepsilon}(p)\cap\C_I)} f(s)\,ds_I\,g(s) = \sum_{j=0}^{n_f -1} a_{-(j+1)}b_{j}. \]
Finally, observe that $a_{-k} = \Res_{p}\left(f_I(s)(s-p)^{k-1}\right)$ and $b_{k} =  \frac{1}{k!} \partial_s^k g_I(p)$ by their definition in \eqref{SeriesCoef}.
\end{proof}

In order to compute the integral in the $S$-functional calculus we denote by $W_c$ the strip $W_c=\{s\in\hh: -c < \Re(s)<c\}$ for $c>0$ and we introduce
 the set $\partial( W_c\cap \mathbb{C}_I)$ for $I\in \mathbb{S}$.
It consists of the two lines $s=c+I\tau$  and $s=-c-I\tau$, $\tau\in \mathbb{R}$,
 and their orientation is such that on $\cc_I$ the orientation of $\partial( W_c\cap \mathbb{C}_I)$ is positive.

\begin{proposition}\label{propuno}
Let $\alpha$ and  $c$ be a real numbers such that
$
\omega<c<|\alpha|.
$
Then for any $u\in \dom (T^2)$ we have
\begin{equation}
\label{ZTInt}\Z_T(t)u=\frac{1}{2\pi}\int_{\partial(W_c\cap \mathbb{C}_I)} e^{ts}(\alpha -s)^{-2}\,ds_I\,S_R^{-1}(s,T) (\alpha \id -T)^2  u.
\end{equation}
\end{proposition}
\begin{proof}
We recall that
$$
S_R^{-1}(s,T)u=\int_0^\infty e^{-ts}\,\Z_T(t)u\,dt,\quad\Re(s) >\omega.
$$
Since $\|\Z_T(t)\| \leq M e^{\omega |t|}$, we get a bound for the $S$-resolvent operator by
\begin{equation}\label{SRUniBound}
\|S_R^{-1}(s,T)u\|=M \int_0^\infty  e^{(\omega -Re(s))t}\|u\|\, dt,\quad\Re(s) >\omega
\end{equation}
which assures that
$\|S_R^{-1}(s,T)\|$ is uniformly bounded on $\{s\in\hh: \Re(s) >\omega +\varepsilon\}$ for any $\varepsilon >0$. A similar consideration gives a uniform bound on
$\{s\in\hh: \Re(s) < -(\omega +\varepsilon)\}$.
Thanks to such bound the integral in \eqref{ZTInt} is well defined since the $(\alpha -s)^{-2}$ goes to zero with order $1/|s|^2$ as $s\to \infty$. We set
$$
F(t)u = \frac{1}{2\pi}\int_{\partial(W_c\cap \mathbb{C}_I)} e^{ts}(\alpha -s)^{-2}\,ds_I\,S_R^{-1}(s,T) (\alpha \id -T)^2  u
$$
for $u\in \dom (T^2)$ and we show that $F(t)u=\Z_T(t)u$ using the Laplace transform.
We first assume $t>0$. If $\Re(p)>c$ then
\begin{align*}
\int_0^\infty e^{-pt} F(t)u \,dt&=
\frac{1}{2\pi}\int_{0}^{\infty}e^{-pt}\int_{\partial(W_c\cap \mathbb{C}_I)}e^{ts}(\alpha -s)^{-2}\,ds_I\,S_R^{-1}(s,T) (\alpha \id -T)^2  u \,dt\\
&=\frac{1}{2\pi}\int_{\partial(W_c\cap \mathbb{C}_I)}\left(\int_{0}^{\infty}e^{-pt}e^{ts}dt \right)(\alpha -s)^{-2}\,ds_I\,S_R^{-1}(s,T) (\alpha \id -T)^2  u.
\end{align*}
Now observe that
$$
\int_0^\infty e^{-pt}e^{ts}\,dt=S_R^{-1}(p,s),
$$
so we have
\begin{align*}
\int_0^\infty  e^{-pt} F(t)u\,dt&=
\frac{1}{2\pi}\int_{\partial(W_c\cap \mathbb{C}_I)} S_R^{-1}(p,s)(\alpha -s)^{-2}\,ds_I\,  S_R^{-1}(s,T)(\alpha \id -T)^2 u.
\end{align*}
We point out that the function $s\mapsto  S_R^{-1}(p,s)(\alpha -s)^{-2}$ is right slice regular for $s\notin[p]\cup\{\alpha\}$ and that the function $s\mapsto S_R^{-1}(s,T)(\alpha\id -T)^2u$ is left slice regular on $\rho_S(T)$.
Observe that the integrand is such that  $(\alpha -s)^{-2}$ goes to zero with order $1/|s|^2$ as $s\to \infty$. By applying Theorem~\ref{CauchyIntAp}, the appropriate version of Cauchy's integral theorem,   we can replace the  path of integration by small negatively oriented circles of radius $\delta>0$ around the singularities of the integrand in the plane $\cc_I$. These singularities are $\alpha$, $p_I = p_0 + I p_1$ and $\overline{p}$ if $I \neq \pm I_p$. We obtain
\begin{align*} \int_0^\infty  e^{-pt}F(t)u\,dt=&-
 \frac{1}{2\pi}\int_{\partial(U_\delta(\alpha)\cap \mathbb{C}_I)}S_R^{-1}(p,s)(\alpha -s)^{-2}\,ds_I\, S_R^{-1}(s,T)(\alpha \id -T)^2u \\
&- \frac{1}{2\pi}\int_{\partial(U_\delta(p_I) \cap\mathbb{C}_I)}S_R^{-1}(p,s)(\alpha -s)^{-2}\,ds_I\, S_R^{-1}(s,T)(\alpha \id -T)^2u \\
&- \frac{1}{2\pi}\int_{\partial(U_\delta(\overline{p_I})\cap \mathbb{C}_I)}S_R^{-1}(p,s)(\alpha -s)^{-2}\,ds_I\, S_R^{-1}(s,T)(\alpha \id -T)^2u
  \end{align*}
Observe that the integrand has a pole of order $2$ at $\alpha$ and poles of order $1$  at $p_I$ and $\overline{p_I}$  (except if $I = \pm I_p$). Applying Lemma \ref{ResMod} with $f(s) = S_R^{-1}(p,s)(\alpha -s)^{-2}$ and $g(s) =  S_R^{-1}(s,T)(\alpha\id -T)^2 u$ yields therefore
\begin{align*}
\int_0^\infty e^{-pt} F(t)u\,dt=& - \Res_{\alpha}\left(S_R^{-1}(p,s)(\alpha-s)^{-2}  \right)S_R^{-1}(\alpha,T)(\alpha\id -T)^2u\\
&-\Res_{\alpha}\left(S_R^{-1}(p,s)(s-\alpha)^{-1}\right)\left(\partial_s S_R^{-1}(\alpha,T)(\alpha\id -T)^2u\right)\\
&- \Res_{p_I}\left(S_R^{-1}(p,s)(\alpha-s)^{-2}\right) S_R^{-1}(p_I,T)(\alpha\id -T)^2u \\
&-\Res_{\overline{p_I}}\left(S_R^{-1}(p,s)(\alpha-s)^{-2}\right)  S_R^{-1}(\overline{p_I},T) (\alpha\id -T)^2u.
\end{align*}
We calculate the residues of the function $f(s)=S_R^{-1}(p,s)(\alpha-s)^{-2}$.  Since it has a pole of order two at $\alpha$, we have
\[ \Res_{\alpha}(f_I) = \lim_{\C_I\ni s\to\alpha} \frac{\partial}{\partial s}f_I(s)(s-\alpha)^2 = \lim_{\C_I\ni s\to\alpha}\frac{\partial}{\partial s} S_R^{-1}(p,s)=\lim_{\C_I\ni s\to\alpha} S_R^{-2}(p,s) = S_R^{-2}(p,\alpha)
\]
and
\[\Res_{\alpha}(f_I(s)(s-\alpha))  = \lim_{\C_I\ni s\to\alpha} f_I(s) (s-\alpha)^2= S_R^{-1}(p,\alpha).\]
The point $p_I = p_0 + I p_1$ is a pole of order $1$. Thus, setting $s_{I_p} = s_0 + I_p s_1\in \C_{I_p}$ for $s = s_0 + I s_1\in\C_I$, we deduce from the Representation formula (see Theorem \ref{RepFo})
\begin{align*}\Res_{p_I}(f_I)& = \lim_{\C_I\ni  s\to p_I}f_I(s)(s-p_I) = \lim_{\C_I\ni s\to p_I} S_R^{-1}(p,s) (\alpha-s)^{-2}(s-p_I) \\
&= \lim_{\C_I\ni s\to p_I} \left[S_R^{-1}(p,s_{I_p}) (1-I_pI)\frac12 + S_R^{-1}(p,\overline{s_{I_p}})  (1+I_pI)\frac12\right] (s-p_I)(\alpha-s)^{-2}\\
&=\left[ \lim_{\C_I\ni s\to p_I} (p-s_{I_p})^{-1} (1-I_pI)(s-p_I) + \lim_{\C_I\ni s\to p_I} (p-\overline{s_{I_p}})^{-1}(1 + I_pI)(s-p_I)\right]\frac12 (\alpha-p_I)^{-2}\\
&=\left[\lim_{\C_I\ni s\to p_I} (p-s_{I_p})^{-1}(1-I_pI)(s-p_I)\right]\frac12 (\alpha-p_I)^{-2}.
\end{align*}
We calculate
\begin{align*}
&\lim_{\C_I\ni s\to p_I} (p-s_{I_p})^{-1}(1-I_pI)(s-p_I) \\
=&\lim_{\C_I\ni s\to p_I}(p-s_{I_p})^{-1}(1-I_pI)(s_0 - p_0) + (p-s_{I_p})^{-1}(1-I_pI)I(s_1 - p_1)\\
=&\lim_{\C_I\ni s\to p_I}(p-s_{I_p})^{-1}(s_0 - p_0)(1-I_pI) + (p-s_{I_p})^{-1}(s_1 - p_1)(I+I_p)\\
=&\lim_{\C_I\ni s\to p_I}(p-s_{I_p})^{-1}(s_0 - p_0)(1-I_pI) + (p-s_{I_p})^{-1}(s_1 - p_1)I_p(-I_pI+1)\\
=&\lim_{\C_I\ni s\to p_I}(p-s_{I_p})^{-1}(s_0 - p_0+ I_p(s_1-p_1))(1-I_pI) \\
=&\lim_{\C_I\ni s\to p_I}(p-s_{I_p})^{-1}(s_{I_p}-p)(1-I_pI) = - (1-I_pI) \\
\end{align*}
and finally obtain
\[ \Res_{p_I}(f_I) = - \frac12 (1-I_pI)(\alpha-p_I)^{-2}.\]
Replacing $I$ by $-I$ in this formula yields
\[ \Res_{\overline{p_I}}(f_I) = - \frac12(1+I_pI)(\alpha-\overline{p_I})^{-2} .\]
Note that these formulas also hold true if $I = \pm I_p$. In this case either $\Res_{p_I}(f_I) = - (\alpha-p_I)^{-2}$ and $\Res_{\overline{p_I}}(f_I) = 0$ because $\overline{p_I}$ is a removable singularity of $f_I$ or vice versa. Moreover,
\[S_R^{-1}(\alpha,T)(\alpha\id -T)^2u = (\alpha\id  - T)^{-1}(\alpha\id -T)^2u = (\alpha\id-T)u\]
 and
\[{\partial_s S_R^{-1}(\alpha,T)(\alpha\id -T)^2u =} {- S_R^{-2}(\alpha,T)(\alpha\id -T)^2u =} {-(\alpha\id  -T)^{-2}(\alpha\id -T)^2u} = -u\]
 because $\alpha$ is real. Putting these pieces together, we get
 \begin{align*}
\int_0^\infty e^{-pt} F(t)u\, dt=&  -S_R^{-2}(p,\alpha)S_R^{-1}(\alpha,T)(\alpha\id -T)^2u + S_R^{-1}(p,\alpha)S_R^{-2}(\alpha,T)(\alpha\id -T)^2u\\
&+ \frac12(1-I_pI)(\alpha-p_I)^{-2}S_R^{-1}(p_I,T)(\alpha\id -T)^2u \\
&+ \frac12(1+I_pI)(\alpha-\overline{p_I})^{-2}S_R^{-1}(\overline{p_I},T)(\alpha\id -T)^2u=\\
=&- (p-\alpha)^{-2}(\alpha\id -T)u + (p-\alpha)^{-1}u+ (p-\alpha)^{-2}S_R^{-1}(p,T)(\alpha\id -T)^{2}u,
 \end{align*}
where that last identity follows from the Representation Formula (see Theorem \ref{RepFo}) because the mapping
\[p\mapsto (\alpha-p)^{-2}S_R^{-1}(p,T)(\alpha\id -T)^2u\]
 is left slice regular. We factor out  $(p-\alpha)^{-2} $ on the left and obtain
 \begin{align*}
\int_0^\infty e^{-pt} F(t)u\, dt= &  (p-\alpha)^{-2}\left(-(\alpha\id -T)u + (p -\alpha )u+ S_R^{-1}(p,T)(\alpha\id -T)^{2}u\right)\\
= &  (p-\alpha)^{-2}\left(pu -2\alpha u +Tu + S_R^{-1}(p,T)(\alpha\id -T)^{2}u\right).
 \end{align*}
Recall that we assumed that $u\in\dom(T^2)$. Hence, $Tu \in \dom(T)$ and so we can apply the right $S$-resolvent \eqref{reseqR} equation twice to obtain
\begin{align*}&S_R^{-1}(p,T)(\alpha\id -T)^{2}u=S_R^{-1}(p,T)(T^2u - 2\alpha Tu + \alpha^2u )\\
&=pS_R^{-1}(p,T)Tu-Tu - 2\alpha pS_R^{-1}(p,T) u+2\alpha u + \alpha^2S_R^{-1}(p,T)u\\
&=p^2S_R^{-1}(p,T)u-pu -Tu - 2\alpha pS_R^{-1}(p,T) u+2\alpha u + \alpha^2S_R^{-1}(p,T)u\\
&=(p-\alpha)^2S_R^{-1}(p,T)u-pu  +2\alpha u -Tu.
\end{align*}
So finally
 \begin{align*}
\int_0^\infty e^{-pt} F(t)u \,dt= &  (p-\alpha)^{-2} (p-\alpha)^2S_R^{-1}(p,T)u= S_R^{-1}(p,T)u.
 \end{align*}
 Hence,
 \[\int_0^\infty e^{-pt}F(t)u \,dt = S_R^{-1}(p,T)u = \int_{0}^{\infty}e^{-pt}\,\Z_T(t)u\,dt, \]
  for $\Re(p)>c$, which implies $F(t)u = \Z_T(t)u$ for $u\in D(T^2)$ and $t\geq0$.

Applying the same reasoning to the semigroup $(\Z(-t))_{t \geq 0}$, with infinitesimal generator $-T$, we see that
\begin{align*}
 \Z(-t)u& = \frac{1}{2\pi}\int_{\partial(W_c\cap \mathbb{C}_I)}e^{ts}(\alpha -s)^{-2} \,ds_I\, S_R^{-1}(s,-T)  (\alpha \id +T)^2 u\\
 &= \frac{1}{2\pi}\int_{\partial(W_c\cap \mathbb{C}_I)} e^{-ts}(\alpha + s)^{-2} \,ds_I\, S_R^{-1}(s,T)(\alpha \id +T)^2u,
 \end{align*}
where the second equality follows by substitution  of $s$ by $-s$ because $-S_R^{-1}(-s,-T) = -S_R^{-1}(s,T)$. Replacing $\alpha$ by $-\alpha$ and $-t$ by $t$, we finally find
\begin{align*}
 \Z(t)u
 &= \frac{1}{2\pi}\int_{\partial(W_c\cap \mathbb{C}_I)}e^{ts}(\alpha -s)^{-2}\,ds_I\,S_R^{-1}(s,T)  (\alpha \id -T)^2  u
 \end{align*}
 also for $t<0$.

\end{proof}

\begin{proposition}\label{propdue}
Let $\alpha$ and $c$ be real numbers such that $\omega<c<|\alpha|$. If $f\in\mathbf{V}(T)$ is right slice regular on $\overline{W_c}$, then for any $u\in D(T^2)$ we have
\begin{equation}
f(T)u=\frac{1}{2\pi}\int_{\partial(W_c\cap \mathbb{C}_I)} f(s)(\alpha -s)^{-2}\,ds_I\, S_R^{-1}(s,T)(\alpha \id -T)^2 u.
\end{equation}
\end{proposition}
\begin{proof}
We recall that $f$ can be represented as
$$
f(s)=\int_\mathbb{R}d\mu(t)\,e^{-st}
$$
with $\mu\in\mathbf{S}(T)$. Using Proposition \ref{propuno} we obtain
\begin{align*}
&\frac{1}{2\pi}\int_{\partial(W_c\cap \mathbb{C}_I)}f(s)(\alpha -s)^{-2}\,ds_I\,S_R^{-1}(s,T)(\alpha \id -T)^2 u\\
=&\frac{1}{2\pi}\int_{\partial(W_c\cap \mathbb{C}_I)}\int_\mathbb{R}d\mu(t)\,e^{-st} (\alpha -s)^{-2}\,ds_I\, S_R^{-1}(s,T)(\alpha \id -T)^2 u\\
=&\int_\mathbb{R}d\mu(t)
 \left(\frac{1}{2\pi}\int_{\partial(W_c\cap \mathbb{C}_I)}e^{-st}(\alpha -s)^{-2}\,ds_I\, S_R^{-1}(s,T) (\alpha \id -T)^2  u\right) \\
=&\int_\mathbb{R}d\mu(t)\, \Z_T(-t)u=f(T)u.
\end{align*}
Note that Fubini's theorem allows us to exchange the order of integration as the $S$-resolvent $S_{R}^{-1}(s,T)$ is uniformly bounded on $\partial(W_c\cap\C_I)$ because of \eqref{SRUniBound}  and so  there exists a constant $K>0$ such that
\begin{align*}
&\frac{1}{2\pi}\int_{\partial(W_c\cap \mathbb{C}_I)}\int_\mathbb{R}\left\|d\mu(t)\,e^{-st} (\alpha -s)^{-2}\,ds_I\, S_R^{-1}(s,T)(\alpha \id -T)^2 u\right\| \\
\leq&\frac{1}{2\pi}\int_{\partial(W_c\cap \mathbb{C}_I)}\int_\mathbb{R}d|\mu|(t)\,e^{-\Re(s)t} \frac{1}{|\alpha -s|^{-2}} \|S_R^{-1}(s,T)\| \|(\alpha \id -T)^2 u\| ds\\
\leq&K\int_{\partial(W_c\cap \mathbb{C}_I)}\int_\mathbb{R}d|\mu|(t)\,e^{c|t|} \frac{1}{(1+|s|)^{2}}\, ds.
\end{align*}
This integral is finite because the fact that $f$ is right slice regular on $\overline{W_c}$ implies
\[\int_\mathbb{R}d|\mu|(t)\,e^{c | t|} <\infty.\]

\end{proof}

\begin{theorem}\label{CompSCalc}
Let $f\in \mathbf{V}(T)$ and suppose that $f$ is right slice regular at infinity.
 Then the operator $f(T)$ defined using the Laplace transform equals the operator $f[T]$ obtained from the $S$-functional calculus.
\end{theorem}
\begin{proof}
Recall that we denote functions of operators obtained by the $S$-functional calculus with square brackets in order to distinguish them from those obtained by the calculus of the present paper. The $S$-functional calculus for unbounded operators satisfies
$$
g[T]=g(\infty)\id +\int_{\partial(W_c\cap \mathbb{C}_I)}g(s)\,ds_I\, S_R^{-1}(s,T).
$$
Consider $\alpha\in\R$ with $c<|\alpha|$ and observe that the function
$$
g(s):=f(s)(\alpha-s)^{-2}
$$
is right slice regular and satisfies $\lim_{s\to \infty}g(s)=0$ since $f$ is right slice regular at infinity. By the theorem of the product, we get
$$
{f}[T](\alpha \id -T)^{-2}u=\frac{1}{2\pi}\int_{\partial(W_c\cap \mathbb{C}_I)} f(s)(\alpha-s)^2\,ds_I\,S_R^{-1}(s,T) u.
$$
But by Proposition \ref{propdue}, it is
$$
f(T)u=\frac{1}{2\pi}\int_{\partial(W_c\cap \mathbb{C}_I)}f(s)(\alpha-s)^2 \,ds_I\, S_R^{-1}(s,T)(\alpha \id -T)^{2}u.
$$
for $u\in D(T^2)$ and so
$$
f[T]u=f(T)u, \qquad\text{ for }u\in D(T^2).
$$
Since  $D(T^2)$ is dense in $V$ and since the operators $f[T]$ and $f(T)$ are bounded we get $f[T]=f(T)$.

\end{proof}

\section{The inversion of the operator $f(T)$}

To study the inversion of an operator we consider a sequence of quaternionic polynomials $P_n(s)$ such that
$\lim_{n\to\infty}P_n(s)f(s)=1$. We point out that in general the pointwise product $P_n(s)f(s)$ is not slice regular and therefore we must limit ourselves to the case that $f$ is an intrinsic function.
The main goal of this section is to deduce sufficient conditions such that
 $$
 \lim_{n\to\infty}P_n(T)f(T)u=u, \ \ \text{ for  every $u\in V$}.
 $$

\begin{lemma}\label{TnDense}
Let $T\in\closOP (V)$ such that $\rho_S(T)\cap\R\neq\emptyset$. Then $\dom (T^n)$ is dense in $V$ for every $n\in\nn$.
\end{lemma}
\begin{proof}
If $\alpha\in\rho_S(T)\cap\R$, then $\dom (T^n) = \dom ((\alpha\id -T)^n) = (\alpha\id -T)^{-n}V = S_R^{-n}(\alpha,T)V$. Therefore
a continuous right linear functional $u^*\in V^*$ on $V$ vanishes on $\dom (T^n)$ if and only if the functional $u^* S_R^{-n}(\alpha,T)$ defined by $\left\langle u^*S_R^{-n}(\alpha,T),v\right\rangle = \left\langle u^*,S_R^{-n}(\alpha,T)v\right\rangle$ vanishes on the entire space $V$.

Now assume that $u^*S_R^{-1}(\alpha,T) =0$  for some $u^*\in V^*$. Then $u^*$ vanishes on $\dom (T)$ and since $\dom (T)$ is dense in $V$ we deduce $u^* = 0$. Since $u^* S_R^{-n}(s,T) = (u^*S_R^{-(n-1)}(s,T))S_R^{-1}(s,T)$ we obtain by induction that even the fact that $u^*S_R^{-n}(s,T) = 0$ for arbitrary $n\in\nn$ implies $ u^* = 0$.

Putting together these two observations, we see that $u^*\dom (T^n) = u^*S_R^{-n}(s,T)V = 0$ implies $u^* = 0$. By the quaternionic version of the Hahn-Banach Theorem (see for example Theorem 4.10.1 in \cite{MR2752913}) $\dom (T^n)$ is dense in $V$.

\end{proof}

\begin{lemma}\label{PolyLemma}
Let $P$ be an intrinsic polynomial of degree $m$ and let $f$ and $P_nf$ both belong to $\mathbf{V}(T)$.
Then $f(T)V\subseteq \dom (T^m)$ and
\[
P[T]f(T)u=(Pf)(T)u, \quad\text{ for  all } u\in V.
\]
\end{lemma}
\begin{proof}
We first consider the case $x\in\dom (T^{m+2})$. Let $\alpha,c\in\R$ with $w<c<|\alpha|$ and let $I\in\S$. The function $Pf$ is the product of two intrinsic functions and therefore intrinsic itself. By \textcolor{red}{\bf} Proposition~\ref{propdue}, Lemma~\ref{PFCLT} and Remark~\ref{SRCompatible}, we have
\begin{equation*}
(\alpha\id  - T)^{-m}(Pf)(T) u = \frac{1}{2\pi}\int_{\partial(W_c\cap\C_I)}(\alpha-s)^{-m}P(s)f(s)(\alpha-s)^{-2}\,ds_I\,S_R^{-1}(s,T)(\alpha\id  - T)^2 u.
\end{equation*}
Write the polynomial $P$ in the form $P(s) = \sum_{k=0}^m a_k (\alpha-s)^k$  with $a_k\in\R$. In view of Proposition~\ref{propdue}, Lemma~\ref{PFCLT} and Remark~\ref{SRCompatible} we obtain again
\begin{gather*}
(\alpha\id  - T)^{-m}(Pf)(T) u \\
 = \sum_{k=0}^ma_k\frac{1}{2\pi}\int_{\partial(W_c\cap\C_I)}(\alpha-s)^{-m+k}f(s)(\alpha-s)^{-2}\,ds_I\,S_R^{-1}(s,T)(\alpha\id  - T)^2 u\\
 = \sum_{k=0}^ma_k(\alpha\id -T)^{-m+k}f(T)u= (\alpha\id -T)^{-m} \sum_{k=0}^ma_k(\alpha\id -T)^{k}f(T)u\\
 = (\alpha\id -T)^{-m}P[T]f(T)u.
\end{gather*}
Consequently, $(Pf)(T)u = P[T]f(T)u$ for $u\in\dom (T^{m+2})$.

Now let $u\in V$ be arbitrary. Since $\dom (T^{m+2})$ is dense Lemma~\ref{TnDense}, there exists a sequence $u_n\in\dom (T^{m+2})$ with $\lim_{n\to\infty}u_n  =u$. Then $f(T)u_n \to f(T) u$ and $P[T]f(T)u_n = (Pf)(T)u_n \to (Pf)(T)u$ as $n\to\infty$. Since $P[T]$ is closed on $\dom (T^m)$, it follows that $f(T)u\in\dom (T^m)$ and $P[T]f(T)u = (Pf)(T)u$.

\end{proof}

\begin{definition}{\rm
A sequence of intrinsic polynomials  $(P_n)_{n\in\nn}$ is called an inverting sequence for an intrinsic function $f\in  \mathbf{V}(T)$ if
\begin{enumerate}[(i)]
\item $P_nf\in \mathbf{V}(T)$,

\item $|P_n(s)f(s)|\leq M,$ $n\in\nn$ for some constant $M>0$ and $\lim_{n\to\infty}P_n(s)f(s)=1$ in a strip $|\Re(s)|\leq \omega+\varepsilon$,

\item $\|(P_nf)(T)\|\leq M$, $n\in \mathbb{N}$ for some constant $M>0$.
\end{enumerate}}
\end{definition}

\begin{theorem}\label{InvID}
If $(P_n)_{n\in\nn}$ is an inverting sequence for an intrinsic function $f\in\mathbf{V}(T)$, then
\[\lim_{n\to\infty}P_n[T]f(T) u = u \qquad\forall u\in V.\]
\end{theorem}
\begin{proof}
First consider $u\in\dom (T^2)$ and choose $\alpha\in\rr$ with $\omega<|\alpha |$. Then Proposition~\ref{propdue} and  Lemma~\ref{PolyLemma} imply
\[P_n[T]f(T)u = (P_nf)(T)u = \frac{1}{2\pi}\int_{\partial(W_{c_n}\cap\C_I)}P_n(s)f(s)(\alpha-s)^{-2}\,ds_I\, S_R^{-1}(s,T) (\alpha\id -T)^2 u\]
for arbitrary $I\in\S$ and $c_n\in\R$ with $w < c_n < |\alpha|$ such that $P_nf$ is right slice regular on $\overline{W_{c_n}}$.
However, we have assumed that there exists a constant $M$ such that $|P_n(s)f(s)|\leq M$ for any $n\in\nn$  on a strip $-(\omega+\varepsilon)\leq\Re(s) \leq \omega+\varepsilon$. Moreover, because of \eqref{SRUniBound}, the right $S$-resolvent is uniformly bounded on any set $\{s\in\C_I: |\Re(s)|>\omega+\varepsilon'\}$ with $\varepsilon'>0$. Applying Cauchy's integral theorem we can therefore replace $\partial(W_{c_n}\cap\C_I)$ for any $n\in\nn$ by $\partial(W_c\cap\C_I)$ where $c$ is a real number with $\omega < c < \min\{|\alpha|, \omega + \varepsilon\}$. In particular, we can choose $c$ independent of~$n$. Lebesgue's dominated convergence theorem allows us to exchange limit and integration and we obtain
\[ P_n[T]f(T) u = \frac{1}{2\pi}\int_{\partial(W_c\cap\C_I)}(\alpha-s)^{-2}\,ds_I\,S_R^{-1}(s,T) (\alpha\id -T)^2 u = u.\]

If $u\in V$ does not belong to $\dom (T^2)$, then we can choose for any $\varepsilon > 0$ a vector $u_{\varepsilon}\in \dom (T^2)$ with $\|u - u_{\varepsilon}\| < \varepsilon$. Since the mappings $(P_nf)(T)$ are uniformly bounded by a constant $M>0$, we get
\begin{gather*}
\| (P_nf)(T) u - u\| \leq  \| (P_nf)(T) u - (P_nf)(T) u_{\varepsilon}\| + \|(P_nf)(T) u_{\varepsilon} - u_{\varepsilon}\| + \| u_{\varepsilon} - u\| \\
\leq M \| u - u_{\varepsilon}\| + \|(P_nf)(T) u_{\varepsilon} - u_{\varepsilon}\| + \| u_{\varepsilon} - u\| \overset{n\to\infty}{\longrightarrow} M \| u - u_{\varepsilon}\|   + \| u_{\varepsilon} - u\| \leq (M+1)\varepsilon.
\end{gather*}
Since $\varepsilon >0$ was arbitrary, we deduce $\lim_{n\to\infty}\| (P_nf)(T) u - u\| = 0$ even for arbitrary $u\in V$.

\end{proof}

\begin{corol}
Let $V$ be reflexive and let $P_n$ be an inverting sequence for an intrinsic function $f\in\mathbf{S}(T)$. A vector $u$ belongs to the range of $f(T)$ if and only if $x$ is in $\dom (P_n[T])$ for all $n\in\nn$ and the sequence $(P_n[T]u)_{n\in\nn}$ is bounded.
\end{corol}
\begin{proof}
If $u\in\ran f(T)$ with $u = f(T)v$ then Lemma~\ref{PolyLemma} implies $u\in\dom( P_n(T))$ for all $n\in\nn$.  Theorem~\ref{InvID} states $\lim_{n\to\infty}P_n[T]u = v$ which implies that $(P_n[T]u)_{n\in\nn}$  is bounded.

To prove the converse statement consider $u \in V$ such that $(P_n[T]u)_{n\in\nn}$ is bounded. Since $V$ is reflexive the set $\{P_n[T]u : n\in\nn\}$ is weakly sequentially compact (the proof that a set $E$ in a reflexive quaternionic Banach space $V$ is weakly sequentially compact if and only if $E$ is bounded can be completed just as in the classical case when $V$ is a complex Banach space, see, e.g., Theorem II.28 in \cite{ds}) and hence there exists a subsequence $(P_{n_k}[T]u)_{k\in\nn}$ and a vector $v\in V$ such that  $\langle x^*, P_{n_k}[T]u\rangle \to \langle x^*, v\rangle$ as $k\to\infty$ for any $x^*\in V^*$. We show $u = f(T)v$.

For any functional $x^*\in V^*$ the mapping $x^*f(T)$ defined by $\langle x^*f(T), w\rangle = \langle x^*,f(T)w\rangle$ also belongs to $V^*$. Hence,
\[ \langle x^*,f(T) P_{n_k}[T]u\rangle = \langle x^*f(T), P_{n_k}[T]u\rangle \to \langle x^*f(T), v\rangle = \langle x^*,f(T)v\rangle.\]
Recall that the measure $\mu$ is real-valued since $f$ is intrinsic. Therefore it commutes with the operator $P_{n_k}[T]$. Recall also that
 if $w\in\dom (T^n)$ for some $n\in\nn$  then $\Z(t)w\in \dom (T^n)$ for any $t\in\R$ and $\Z(t)T^nw = T^n\Z(t)w$. Thus, $P_{n_k}[T]\Z(t) u = \Z(t)P_{n_k}[T] u$ because $P_{n_k}$ has real coefficients. Moreover, we can therefore exchange the integral with the unbounded operator $P_{n_k}[T]$ in the following computation
\begin{align*}
f(T)P_{n_k}[T]u = \int_{\rr} d\mu(t)\, \Z(-t) P_{n_k}[T]u = P_{n_k}[T]\int_{\rr} d\mu(t)\, \Z(-t) u = P_{n_k}[T]f(T)u.
\end{align*}
Theorem~\ref{InvID} implies for any $x^*\in V^*$
\[ \langle x^*,u\rangle  = \lim_{k\to\infty} \langle x^*, P_{n_k}[T]f(T)u\rangle = \lim_{k\to\infty} \langle x^*, f(T) P_{n_k}[T]u\rangle = \langle x^*, f(T) v\rangle\]
and so $u = f(T)v$ follows from the quaternionic version of the Hahn-Banach Theorem (see for example Theorem 4.10.1 in \cite{MR2752913}).

\end{proof}

\end{document}